\documentclass[12pt]{article}
\usepackage{amssymb,amsthm,hyperref,amsmath,graphicx,bm,caption,color,booktabs,comment,stmaryrd, threeparttable, subcaption, float, mathtools, float, epstopdf,mathdots}
\usepackage[numbers,sort&compress]{natbib}
\numberwithin{figure}{section}
\numberwithin{table}{section}
\newtheorem{thm}{Theorem}[section]

\newtheorem{prop}{Proposition}[section]
\newtheorem{example}{Example}
\newtheorem{defn}{Definition}[section]
\theoremstyle{remark}
\newtheorem{rem}{Remark}[section]

\newcommand{\ie}{{\it{i.e.}}}

\newcommand{\eg}{{\it{e.g.}}}

\newcommand{\diag}{\mbox{diag}}

\title{Skew-Symmetric Formulations and Energy-Stable Schemes for Compressible Reactive Flow
	Equations}
\author{
	Junbo Gong\thanks{E-mail: dstnyg@outlook.com}\\
	\small{\textit{Department of Applied Mathematics}},\\
	\small{\textit{University of Science and Technology Beijing, Beijing 100083, P.R.China}}
	\and
	Weifeng Zhao\thanks{Corresponding author. E-mail: wfzhao@ustb.edu.cn}\\
	\small{\textit{Department of Applied Mathematics}},\\
	\small{\textit{University of Science and Technology Beijing, Beijing 100083, P.R.China}}}

\date{}

\begin{document}
	\maketitle{}
	
	\begin{abstract}
		This paper develops skew-symmetric formulations and energy stable finite difference schemes for the reactive Euler and reactive Navier--Stokes equations. By introducing suitable transformed variables, we rewrite the governing equations into skew-symmetric forms and obtain continuous energy estimates by standard integration-by-parts arguments. The proposed formulations extend the skew-symmetric framework for compressible flow equations to reactive systems, where the chemical energy and the species equation must be incorporated into the transformed variables. Based on these formulations, we construct semi-discrete schemes with summation-by-parts (SBP) operators. The convective terms are discretized by the non-dissipative part of linear upwind SBP operators, while a nonlinear SBP--WENO operator is used only to supply artificial dissipation near discontinuities. For periodic boundary conditions, the resulting schemes satisfy semi-discrete energy estimates consistent with the continuous analysis. Numerical experiments in one and two space dimensions validate the designed order of accuracy and demonstrate stable performance for problems with shocks, viscous layers, detonation-type waves, and complex wave interactions.
		
	\end{abstract}
	
	\noindent\textbf{Keywords:} reactive flow equations; skew-symmetric formulation; SBP operators.

	\section{Introduction}
	
	Reactive flow equations are fundamental models for compressible flows with chemical reactions, including combustion, detonation, explosion dynamics, and propulsion systems involving chemically reacting media \cite{Williams2018CombustionTheory,PoinsotVeynante2005TNC,FickettDavis2000Detonation,Kailasanath2000DetonationPropulsion}. Their numerical simulation is difficult because convection, diffusion, heat release, and species transport may interact over very different spatial and temporal scales. In detonation and shock-dominated regimes, the solutions may contain strong discontinuities and sharp thermodynamic gradients; at the same time, the chemical source terms and the positivity of physical quantities such as density, pressure, and mass fraction have to be treated carefully. These features make stability a central issue in the design of reliable high-order methods for reactive flows.
	
	Many numerical methods have been developed for reactive flow problems. Gas-kinetic schemes were proposed for multimaterial flows with chemical reactions in \cite{LianXu2000GasKineticReaction}. For detonation waves, high-resolution shock-capturing methods have played an important role; see, for example, the numerical study of pulsating detonations in \cite{HwangFedkiwOsher2000PulsatingDetonation}, high-order WENO computations of two-dimensional detonation waves in \cite{GaoDonLi2012WENO}, and adaptive finite volume simulations on unstructured meshes in \cite{Hu2017DetonationUnstructured}. Other related developments include an HLLC-type method for reactive flows with general equations of state \cite{ZengXiaoNi2016GeneralEOS}, positivity-preserving schemes for production--destruction systems arising in nonequilibrium flows \cite{HuangZhaoShu2019PositivityPD}, front-capturing methods for the reactive Euler equations \cite{XiaoNiWangYang2021FrontCapturing}, and conservative discontinuous Galerkin (DG) methods for the chemically reacting Navier--Stokes equations \cite{JohnsonKercher2020ConservativeDGReactingNS}. These works have greatly improved the accuracy and robustness of numerical simulations for reacting compressible flows.
	
	Another important direction is the construction of schemes satisfying entropy or entropy-like stability properties. For the reactive Euler equations with a single irreversible reaction, it was shown in \cite{Zhao2022EntropyReactiveEuler} that the classical thermodynamic entropy is not strictly convex, and a strictly convex mathematical entropy was introduced to construct entropy preserving and entropy stable finite difference schemes. This idea was further used in \cite{ZuoZhaoLin2024EntropyStableDGReactiveEuler} to design a positivity preserving and oscillation-free entropy stable DG method, which is entropy stable for both the thermodynamic entropy and the strictly convex mathematical entropy. For chemically reacting compressible Euler flows with complex thermodynamics, positivity-preserving and entropy-bounded DG methods were developed in \cite{ChingJohnsonKercher2024EntropyBoundedDGPartI,ChingJohnsonKercher2024EntropyBoundedDGPartII}; the corresponding extension to viscous chemically reacting Navier--Stokes equations was presented in \cite{ChingJohnsonBurrowsHiggsKercher2025EntropyBoundedDGNS}. Related entropy-based high-order methods have also been studied for hypersonic flows with real chemistry \cite{PeyvanShuklaChanKarniadakis2023HypersonicRealChemistry} and high-enthalpy gas flows with internal energy modes \cite{OblapenkoTorrilhon2025EntropyConservativeHighEnthalpy}.
	
	In this paper we follow a different route, namely the skew-symmetric energy method. Skew-symmetric and symmetry-preserving formulations have appeared in several contexts. Vabishchevich derived $L^2$ estimates for transformed hydrodynamic variables \cite{Vabishchevich2007Hydrodynamics}, while conservative or symmetry-preserving skew-symmetric discretizations were investigated for compressible turbulence, aeroacoustics, and plasma moment equations in \cite{RozemaKokVerstappenVeldman2014SymmetryPreserving,ReissSesterhenn2014ConservativeSkewSymmetricNS,HalpernWaltz2018AntisymmetricPlasma}. A systematic stability theory was developed by Nordstr\"om for nonlinear initial-boundary value problems in skew-symmetric form \cite{Nordstrom2022PrimalDual}. This framework was applied to the shallow water equations \cite{NordstromWinters2022}, to the compressible Euler equations \cite{Nordstrom2022SkewEuler}, and then to the compressible Navier--Stokes equations \cite{Nordstrom2024SkewNS}. Boundary treatments for nonlinear problems were further studied in \cite{Nordstrom2024NonlinearBC,Nordstrom2025OpenBC,Nordstrom2026LinearNonlinearBC}, and an extension to incompressible multiphase flow was given in \cite{NordstromMalan2025Multiphase}. In these works, the essential point is that a suitable change of variables and a skew-symmetric splitting of the differential operator yield an energy estimate, and the same structure can be preserved discretely by summation-by-parts (SBP) operators \cite{SvardNordstrom2014SBPReview,FernandezHickenZingg2014SBPSATReview}.
	
	Motivated by these developments, we extend the skew-symmetric framework to reactive flow equations. The extension is not a direct repetition of the non-reactive compressible Euler or Navier--Stokes analysis. Because the total energy contains the chemical energy \(q\rho Y\), the transformed variables must include both the species variable and the heat release. Moreover, the reaction source term must enter the energy estimate with the correct sign. We show that suitable transformed variables lead to skew-symmetric formulations for both the reactive Euler and reactive Navier--Stokes equations. The resulting formulations give continuous energy estimates by integration by parts, without using entropy-variable chain-rule arguments.
	
	Based on the continuous formulations, we construct semi-discrete finite difference schemes which retain the skew-symmetric structure. The convective part is approximated by the central, non-dissipative component of linear upwind SBP operators, while artificial dissipation is built from the nonlinear upwind SBP--WENO operator of \cite{JiangWang2024UpwindSBP}. This separation is important: the nonlinear WENO mechanism is used to control spurious oscillations near discontinuities, but it is not inserted into the skew-symmetric balance itself. Under periodic boundary conditions, the SBP property then yields a semi-discrete energy estimate consistent with the continuous one.
	
	The main contributions of this work are summarized as follows. First, we derive transformed-variable skew-symmetric formulations for the reactive Euler equations and for the reactive Navier--Stokes equations. Second, we prove continuous energy estimates for these formulations, including the effect of the chemical reaction and the viscous, heat-conductive, and species-diffusive terms. Third, we design SBP finite difference schemes with SBP--WENO artificial dissipation and prove semi-discrete energy stability under periodic boundary conditions. Finally, we present one- and two-dimensional numerical experiments to verify the accuracy of the schemes and to demonstrate their performance for discontinuous and detonation-type problems.
	
	The rest of the paper is organized as follows. In Section~\ref{sec:preliminaries}, we introduce the reactive flow equations and recall the general theory of skew-symmetric formulations. Section~\ref{sec:ns} derives skew-symmetric forms for the reactive Euler and reactive Navier--Stokes equations. Section~\ref{sec:es} constructs the semi-discrete schemes and proves the energy estimates under periodic boundary conditions. Numerical experiments are given in Section~\ref{sec:ne}. Section~\ref{sec:conclusion} concludes the paper.

	\section{Preliminaries}
	\setcounter{equation}{0}
	\label{sec:preliminaries}
	
	This section introduces the reactive flow equations considered in this work and recalls the general framework of \cite{Nordstrom2022PrimalDual} for obtaining skew-symmetric formulations and energy estimates.
	
	\subsection{Reactive flow equations}
	
	We consider the two-dimensional reactive Navier--Stokes equations
	\begin{equation}
		\label{eq:RNS}
		\begin{aligned}
			\partial_t \rho + \nabla \cdot (\rho \mathbf{u}) &= 0,\\
			\partial_t (\rho \mathbf{u}) + \nabla \cdot (\rho \mathbf{u}\mathbf{u} + p\mathbf{I}) &= \nabla \cdot \boldsymbol{\tau},\\
			\partial_t (\rho E) + \nabla \cdot (\rho \mathbf{u} E + p\mathbf{u}) &= \nabla \cdot (\boldsymbol{\tau} \cdot \mathbf{u}) + \nabla \cdot (\kappa \nabla T) + \nabla \cdot (q\sigma \nabla Y),\\
			\partial_t (\rho Y) + \nabla \cdot (\rho \mathbf{u} Y) &= \nabla \cdot (\sigma \nabla Y) + \omega.
		\end{aligned}
	\end{equation}
	Here $\rho$ is the density, $\mathbf{u}=(u_1,u_2)^T$ is the velocity, $p$ is the pressure, $E$ is the total energy per unit mass, $\kappa>0$ is the thermal conductivity, $q>0$ is the heat release, $\sigma>0$ is the diffusion coefficient, $Y \in [0,1]$ is the reactant mass fraction, and $\omega \le 0$ is the reaction term. The stress tensor is
	$
	\boldsymbol{\tau}
	=
	\lambda \nabla \cdot \mathbf{u}\,\mathbf{I}
	+
	\mu\bigl(\nabla \mathbf{u}+(\nabla \mathbf{u})^{\dagger}\bigr),
	$
	where $\mathbf{I}$ is the identity matrix, and $\lambda$ and $\mu$ are the bulk viscosity and dynamic viscosity, respectively.
	The total energy $\rho E$ is given by
	\begin{equation} \label{eq:total_energy}
		\rho E = \frac{1}{2}\rho |\mathbf{u}|^2 + \rho e + q\rho Y
	\end{equation}
	with $\rho e$ the internal energy.
	The system is closed by the ideal-gas equation of state
	\begin{equation} \label{eq:eos}
		p = (\gamma - 1)\rho e,
	\end{equation}
	where $\gamma>1$ is the constant specific heat ratio. Neglecting the viscous stress, the heat conduction, and the species diffusion in \eqref{eq:RNS}, we obtain the reactive Euler equations
	\begin{equation} \label{eq:REuler}
		\begin{aligned}
			\partial_t \rho + \nabla \cdot (\rho \mathbf{u}) &= 0,\\
			\partial_t (\rho \mathbf{u}) + \nabla \cdot (\rho \mathbf{u}\mathbf{u} + p\mathbf{I}) &= 0,\\
			\partial_t (\rho E) + \nabla \cdot (\rho \mathbf{u} E + p\mathbf{u}) &= 0,\\
			\partial_t (\rho Y) + \nabla \cdot (\rho \mathbf{u} Y) &= \omega.
		\end{aligned}
	\end{equation}
	
	\subsection{Main theory}
	
	\subsubsection{Skew-symmetric form of hyperbolic equations}
	
	Next we introduce the stability theory developed in \cite{Nordstrom2022PrimalDual} for general nonlinear hyperbolic systems written in a skew-symmetric form. Consider the following general first-order hyperbolic system
	\begin{equation} \label{eq:general_system}
		P U_t + \sum_{i=1}^d (A_i(V) U)_{x_i} + \sum_{i=1}^d B_i(V) U_{x_i} + W(V) U = 0,
	\end{equation}
	where $t\geq 0$, $x=(x_1,x_2,...,x_d) \in \Omega$, $U,V \in \mathbb R^n$ are smooth vector variables,  $P \in \mathbb R^{n\times n}$ is a constant symmetric positive-definite matrix, and matrices $A_i=A_i(V)$, $B_i=B_i(V)$, and $W=W(V)$ are assumed to be smooth functions of $V$.
	
	Define the energy $\|U\|_P^2 = \int_{\Omega} U^{T} P U \,\mathrm{d}\Omega$. We slightly modify the notion of energy conservation in \cite{Nordstrom2022PrimalDual} as follows.
	
	
	\begin{prop} \label{prop:energy_est}
		If the matrices in \eqref{eq:general_system} satisfy
		\begin{equation}\label{eq:skew_cond}
			B_i = A_i^{T}, \qquad i = 1,\dots,d
		\end{equation}
		and
		\begin{equation}\label{eq:skew_cond2}
			W + W^T \geq 0,
		\end{equation}
		then system \eqref{eq:general_system} with periodic boundary conditions is energy stable in the sense that
		\begin{equation}\label{eq:energy_evol}
			\begin{aligned}
				\frac{1}{2}\frac{\mathrm{d}}{\mathrm{d}t} \|U\|_P^2
				&=
				-
				\int_{\Omega} U^{T} W U \,\mathrm{d}\Omega \leq 0.
			\end{aligned}
		\end{equation}
	\end{prop}
	
	\begin{proof}
		Note that $P$ is symmetric positive definite. Multiplying \eqref{eq:general_system} by \(U^T\) from the left and integrating over the domain \(\Omega\), we obtain
		\[
		\frac{1}{2}\frac{d}{dt}\|U\|_P^2
		+
		\int_{\Omega}
		U^T
		\sum_{i=1}^{d}(A_iU)_{x_i}
		\, d\Omega
		+
		\int_{\Omega}
		U^T
		\sum_{i=1}^{d}B_iU_{x_i}
		\, d\Omega
		+
		\int_{\Omega}
		U^T WU
		\, d\Omega
		=
		0 .
		\]
		Applying integration by parts for the above equation gives
		\[
		\frac{1}{2}\frac{d}{dt}\|U\|_P^2
		+
		\oint_{\partial\Omega}
		U^T
		\left(
		\sum_{i=1}^{d} n_i A_i
		\right)
		U\, ds
		=
		\int_{\Omega}
		\sum_{i=1}^{d}
		\left(
		U_{x_i}^T A_i U
		-
		U^T B_i U_{x_i}
		\right)
		\, d\Omega
		-
		\int_{\Omega}
		U^T WU\, d\Omega .
		\]
		Since \(B_i=A_i^T\), we have
		\[
		U_{x_i}^T A_i U
		=
		U^T B_i U_{x_i}.
		\]
		Moreover, under periodic boundary conditions, the boundary contributions cancel and thereby
		\[
		\frac{1}{2}\frac{d}{dt}\|U\|_P^2
		=
		-
		\int_{\Omega}
		U^T WU\, d\Omega.
		\]
		The right-hand side of the above equation is nonpositive since \(W+W^T\ge0\).
		Thus the system is energy stable.
	\end{proof}

	\begin{rem}
		System \eqref{eq:general_system} satisfying \eqref{eq:skew_cond} and \eqref{eq:skew_cond2} is said to be in a skew-symmetric form \cite{Nordstrom2022PrimalDual}. Note that the notion of energy conservation for system \eqref{eq:general_system} developed in \cite{Nordstrom2022PrimalDual} requires that $W + W^T=0$. Here we relax this condition to nonnegativity so that more equations, such as the reactive Euler equations \eqref{eq:REuler}, can be included in the theory.
	\end{rem}
	
	\begin{rem}
		The main advantage of skew-symmetric formulations is that the nonlinear convective operator is skew-symmetric by construction. Consequently, the energy stability follows directly from standard integration by parts, because the convective contributions cancel in the energy norm without the need for lengthy chain-rule manipulations. This continuous skew-symmetric structure is also essential for numerical approximation, as it enables a direct SBP discretization whose semi-discrete energy estimate mirrors the continuous one.
	\end{rem}
	
	\begin{rem}\label{rem:construction_procedure}
		Proposition~\ref{prop:energy_est} shows that, regardless of the original form of the system, energy stability follows once it admits a skew-symmetric reformulation. Following \cite{Nordstrom2022PrimalDual}, the reformulation of a given system into the form of \eqref{eq:general_system} consists of three steps:
		\begin{enumerate}
			\item Identify a suitable energy norm matrix $P$ and introduce a corresponding set of dependent variables.
			\item Derive the governing equations in these variables from the standard conservative form.
			\item Transform the new set of equations into a skew-symmetric form as in \eqref{eq:general_system} so that conditions \eqref{eq:skew_cond} and \eqref{eq:skew_cond2} hold.
		\end{enumerate}
	\end{rem}

	\subsubsection{Skew-symmetric form of hyperbolic--parabolic systems}
	\label{subsec:general_hp}
	
	Following the work in \cite{Nordstrom2024SkewNS} for the Navier--Stokes equations, we further extend the above framework to the general hyperbolic--parabolic system
	\begin{equation}\label{eq:general_hp_system}
		P U_t
		+\displaystyle\sum_{i=1}^{d}(A_i(V)U)_{x_i}
		+\displaystyle\sum_{i=1}^{d}B_i(V)U_{x_i}
		+W(V)U = F(V,\nabla V,\nabla^2 V).
	\end{equation}
	Here $F(V,\nabla V,\nabla^2 V)$ collects the parabolic contributions. Then we have
	
	\begin{prop}\label{prop:general_hp_energy}
		If the matrices in \eqref{eq:general_hp_system} satisfy \eqref{eq:skew_cond} and \eqref{eq:skew_cond2}, and
		$U^{T}F(V,\nabla V,\nabla^2V)$ can be written in a certain divergence form
		\begin{equation}\label{eq:parabolic_divergence}
			U^{T}F(V,\nabla V,\nabla^2V)
			=
			\sum_{i=1}^{d}
			\bigl(G_i(U,V,\nabla V)\bigr)_{x_i},
		\end{equation}
		then system \eqref{eq:general_hp_system} with periodic boundary conditions satisfies the energy stability \eqref{eq:energy_evol}.
	\end{prop}
	
	
	The proof of the above proposition is direct. To reformulate a system into the skew-symmetric form in \eqref{eq:general_hp_system}, one can follow Steps 1--3 for the hyperbolic system \eqref{eq:general_system} and further check the condition \eqref{eq:parabolic_divergence}.
	
	\section{Skew-symmetric formulations}
	\label{sec:ns}
	\setcounter{equation}{0}
	
	In this section, we apply the general framework of Section~\ref{sec:preliminaries} to derive skew-symmetric formulations for the reactive Euler and reactive Navier--Stokes equations.
	
	\subsection{Skew-symmetric formulation of the reactive Euler equations}
	
	We first follow the procedure summarized in Remark~\ref{rem:construction_procedure} to consider the two-dimensional (2D) reactive Euler equations \eqref{eq:REuler}.
	In contrast to the non-reactive case, the transformation of variables must incorporate the chemical potential energy.
	Motivated by the variables in \cite{Nordstrom2022SkewEuler} for the Euler equations and the structure of the total energy $\rho E = \frac{p}{\gamma-1} + \frac{1}{2}\rho |\mathbf{u}|^2 + q\rho Y$, we define
	\begin{equation} \label{eq:new_vars}
		\Phi =
		\begin{pmatrix}
			\phi_1 \\
			\phi_2 \\
			\phi_3 \\
			\phi_4 \\
			\phi_5
		\end{pmatrix}
		=
		\begin{pmatrix}
			\sqrt{\rho} \\
			\sqrt{\rho}u_1 \\
			\sqrt{\rho}u_2 \\
			\sqrt{p + q\rho Y (\gamma - 1)} \\
			\sqrt{\rho Y}
		\end{pmatrix},
		\quad
		P = \text{diag}(\alpha^2, \beta^2, \theta^2, \eta^2, 1),
	\end{equation}
	where $\alpha^2$, $\beta^2$, $\theta^2$, and $\eta^2$ are positive constants to be determined later. Note that if $Y=0$, $\phi_1, \phi_2, \phi_3, \phi_4$ are the variables introduced in \cite{Nordstrom2022SkewEuler} for the Euler equations without chemical reactions. This completes the first step in Remark~\ref{rem:construction_procedure}.
	
	Using the ideal gas law, the 2D reactive Euler equations in primitive variables take the form
	\begin{equation*}\label{eq:primitive_form}
		\begin{aligned}
			\rho_t + u_1\rho_{x_1} + u_2\rho_{x_2} + \rho(u_{1x_1} + u_{2x_2}) &= 0, \\
			u_{1t} + u_1u_{1x_1} + u_2u_{1x_2} + p_{x_1}/\rho &= 0, \\
			u_{2t} + u_1u_{2x_1} + u_2u_{2x_2} + p_{x_2}/\rho &= 0, \\
			[p + (\gamma - 1)q\rho Y]_t + [(p + (\gamma - 1)q\rho Y)u_1]_{x_1} & \\
			\quad + [(p + (\gamma - 1)q\rho Y)u_2]_{x_2} + (\gamma - 1)p(u_{1x_1} + u_{2x_2}) &= 0, \\
			(\rho Y)_t + u_1(\rho Y)_{x_1} + u_2(\rho Y)_{x_2} + \rho Y(u_{1x_1} + u_{2x_2}) &= \omega.
		\end{aligned}
	\end{equation*}
	With these, we apply the chain rule to obtain the equations for $\Phi$ defined in \eqref{eq:new_vars}:
	\begin{equation}\label{eq:phi_system}
		\begin{aligned}
			(\phi_1)_t
			+ \frac{1}{2}\Big[u_1(\phi_1)_{x_1} + (\phi_2)_{x_1}\Big]
			+ \frac{1}{2}\Big[u_2(\phi_1)_{x_2} + (\phi_3)_{x_2}\Big] &= 0, \\[1em]
			(\phi_2)_t
			+ \frac{1}{2}\biggl[
			- u_1^2(\phi_1)_{x_1}
			+ 3u_1(\phi_2)_{x_1}
			+ \frac{4\phi_4}{\phi_1}(\phi_4)_{x_1}
			- \frac{4q(\gamma - 1)\phi_5}{\phi_1}(\phi_5)_{x_1}
			\biggr] \\
			\qquad
			+ \frac{1}{2}\Big[
			- u_1u_2(\phi_1)_{x_2}
			+ 2u_2(\phi_2)_{x_2}
			+ u_1(\phi_3)_{x_2}
			\Big] &= 0, \\[1em]
			(\phi_3)_t
			+ \frac{1}{2}\Big[
			- u_1u_2(\phi_1)_{x_1}
			+ u_2(\phi_2)_{x_1}
			+ 2u_1(\phi_3)_{x_1}
			\Big] \\
			\qquad
			+ \frac{1}{2}\biggl[
			- u_2^2(\phi_1)_{x_2}
			+ 3u_2(\phi_3)_{x_2}
			+ \frac{4\phi_4}{\phi_1}(\phi_4)_{x_2}
			- \frac{4q(\gamma - 1)\phi_5}{\phi_1}(\phi_5)_{x_2}
			\biggr] &= 0, \\[1em]
			(\phi_4)_t
			- \sum_{i=1}^{2}\frac{\gamma u_i}{2}\frac{\phi_4}{\phi_1}(\phi_1)_{x_i}
			+ \sum_{i=1}^{2}\frac{q(\gamma-1)^2u_i}{2}
			\frac{\phi_5^2}{\phi_4\phi_1}(\phi_1)_{x_i} \\
			\qquad
			+ \sum_{i=1}^{2}\frac{\gamma}{2}\frac{\phi_4}{\phi_1}(\phi_{i+1})_{x_i}
			- \sum_{i=1}^{2}\frac{q(\gamma-1)^2}{2}
			\frac{\phi_5^2}{\phi_4\phi_1}(\phi_{i+1})_{x_i}
			+ \sum_{i=1}^{2}u_i(\phi_4)_{x_i} &= 0, \\[1em]
			(\phi_5)_t
			+ \sum_{i=1}^{2}u_i(\phi_5)_{x_i}
			+ \sum_{i=1}^{2}\frac{1}{2}\frac{\phi_5}{\phi_1}(\phi_{i+1})_{x_i}
			- \sum_{i=1}^{2}\frac{u_i}{2}\frac{\phi_5}{\phi_1}(\phi_1)_{x_i}
			- \frac{\omega}{2\phi_5} &= 0,
		\end{aligned}
	\end{equation}
	which can be written in matrix form
	\begin{equation} \label{eq:quasi_linear}
		\Phi_t + A\Phi_{x_1} + B\Phi_{x_2} + C\Phi = 0.
	\end{equation}
	Here, $A, B \in \mathbb{R}^{5 \times 5}$ denote the convective transport in the $x_1$- and $x_2$-directions, respectively, and $C \in \mathbb{R}^{5 \times 5}$ collects the source terms. Their explicit expressions are given by
	\begin{equation} \label{eq:matrix_ABC}
		\begin{split}
			A &= \frac{1}{2}
			\begin{pmatrix}
				u_1 & 1 & 0 & 0 & 0 \\
				-u_1^2 & 3u_1 & 0 & \frac{4\phi_4}{\phi_1} & -\frac{4q(\gamma-1)\phi_5}{\phi_1} \\
				-u_1 u_2 & u_2 & 2u_1 & 0 & 0 \\
				-\frac{\gamma u_1 p + (\gamma-1)u_1q\rho Y}{\phi_1 \phi_4} &
				\frac{\gamma p + (\gamma-1)q \rho Y}{\phi_1 \phi_4} & 0 & 2u_1 & 0 \\
				-\frac{u_1\phi_5}{\phi_1} & \frac{\phi_5}{\phi_1} & 0 & 0 & 2u_1
			\end{pmatrix}, \\
			B &= \frac{1}{2}
			\begin{pmatrix}
				u_2 & 0 & 1 & 0 & 0 \\
				-u_1 u_2 & 2u_2 & u_1 & 0 & 0 \\
				-u_2^2 & 0 & 3u_2 & \frac{4\phi_4}{\phi_1} & -\frac{4q(\gamma-1)\phi_5}{\phi_1} \\
				-\frac{\gamma u_2 p + (\gamma-1)u_2q\rho Y}{\phi_1 \phi_4} &
				0 & \frac{\gamma p + (\gamma-1)q \rho Y}{\phi_1 \phi_4} & 2u_2 & 0 \\
				-\frac{u_2\phi_5}{\phi_1} & 0 & \frac{\phi_5}{\phi_1} & 0 & 2u_2
			\end{pmatrix},\\
			C &= \text{diag}\left(0, 0, 0, 0, -\frac{\omega}{2\phi_5^2}\right).
		\end{split}
	\end{equation}
	Note that the species equation introduces nontrivial coupling in the last row and column of $A$ and $B$ through terms involving $\phi_5$ and $Y$; these contributions are absent in the non-reactive Euler equations. This completes the second step in Remark~\ref{rem:construction_procedure}.
	
	To further obtain the skew-symmetric form, we seek a decomposition of the convective operator into divergence and nonconservative parts. Specifically, we follow \cite{Nordstrom2022PrimalDual} to determine matrices $\tilde{A}$, $\tilde{B}$, and the norm matrix $P$ such that
	\begin{equation} \label{eq:ansatz}
		\begin{aligned}
			(\tilde{A}\Phi)_{x_1} + \tilde{A}^{T} \Phi_{x_1} &= 2 P A \Phi_{x_1}, \\
			(\tilde{B}\Phi)_{x_2} + \tilde{B}^{T} \Phi_{x_2} &= 2 P B \Phi_{x_2}.
		\end{aligned}
	\end{equation}
	This splitting isolates the skew-symmetric structure and identifies the boundary terms that remain after integration by parts. Then we compare the coefficients of the spatial derivatives on both sides of \eqref{eq:ansatz}, which yields an algebraic system for the unknown entries of $P$, $\tilde{A}$, and $\tilde{B}$.
	Solving this system gives the admissible norm
	\begin{equation} \label{eq:P_sol}
		P = \text{diag}\left(\alpha^2, \frac{\gamma-1}{2}, \frac{\gamma-1}{2}, 1, 1\right),
	\end{equation}
	where $\alpha^2>0$ remains a free parameter. The corresponding matrices $\tilde{A}$ and $\tilde{B}$ are
	\begin{equation} \label{eq:ABtilde}
		\setlength{\arraycolsep}{2.5pt}
		\begin{split}
			\tilde{A} &=
			\begin{pmatrix}
				\alpha^2 u_1 & 0 & 0 & 0 & 0 \\
				0 & \frac{(\gamma-1)u_1}{2} & 0 & 0 & 0 \\
				0 & 0 & \frac{(\gamma-1)u_1}{2} & 0 & 0 \\
				0 & 2(\gamma-1)\frac{\phi_4}{\phi_1} & 0 & \left[(2-\gamma) - \frac{q\rho Y (\gamma-1)^2}{p+q\rho Y(\gamma-1)}\right]u_1 & 0 \\
				0 & -2q(\gamma-1)^2 \frac{\phi_5}{\phi_1} & 0 & 2q u_1 (\gamma - 1)^2\frac{\phi_5}{\phi_4} & u_1
			\end{pmatrix}, \\[1em]
			\tilde{B} &=
			\begin{pmatrix}
				\alpha^2 u_2 & 0 & 0 & 0 & 0 \\
				0 & \frac{(\gamma-1)u_2}{2} & 0 & 0 & 0 \\
				0 & 0 & \frac{(\gamma-1)u_2}{2} & 0 & 0 \\
				0 & 0 & 2(\gamma-1)\frac{\phi_4}{\phi_1} & \left[(2-\gamma) - \frac{q\rho Y (\gamma-1)^2}{p+q\rho Y(\gamma-1)}\right]u_2 & 0 \\
				0 & 0 & -2q(\gamma-1)^2 \frac{\phi_5}{\phi_1} & 2q u_2 (\gamma - 1)^2\frac{\phi_5}{\phi_4} & u_2
			\end{pmatrix}.
		\end{split}
	\end{equation}
	
	With $\tilde{A}$, $\tilde{B}$, and $P$ determined, substitution of \eqref{eq:ansatz} into the quasi-linear form \eqref{eq:quasi_linear} yields the final skew-symmetric formulation of the reactive Euler equations:
	\begin{equation}\label{eq:final_form0}
		P \Phi_t + (\frac{1}{2}\tilde A \Phi)_{x_1} + \frac{1}{2}\tilde A^T \Phi_{x_1}
		+ (\frac{1}{2}\tilde B \Phi)_{x_2} + \frac{1}{2}\tilde B^T  \Phi_{x_2} + PC\Phi = 0,
	\end{equation}
	or equivalently,
	\begin{equation}\label{eq:final_form}
		\Phi_t + (A_1\Phi)_{x_1} + A_2\Phi_{x_1} + (B_1\Phi)_{x_2} + B_2\Phi_{x_2} + C\Phi = 0,
	\end{equation}
	where
	\begin{equation}\label{eq:A1B1}
		\begin{split}
			A_1 = \frac{P^{-1}\widetilde{A}}{2}
			=\frac{1}{2}
			\begin{pmatrix}
				u_1 & 0 & 0 & 0 & 0 \\
				0 & u_1 & 0 & 0 & 0 \\
				0 & 0 & u_1 & 0 & 0 \\
				0 & 2(\gamma-1)\frac{\phi_4}{\phi_1} & 0 & \left[(2-\gamma) - \frac{q\rho Y (\gamma-1)^2}{p+q\rho Y(\gamma-1)}\right]u_1 & 0 \\
				0 & -2q(\gamma-1)^2 \frac{\phi_5}{\phi_1} & 0 & 2q u_1 (\gamma - 1)^2\frac{\phi_5}{\phi_4} & u_1
			\end{pmatrix}, \\
			A_2 = \frac{P^{-1}\widetilde{A}^{T}}{2}
			=\frac{1}{2}
			\begin{pmatrix}
				u_1 & 0 & 0 & 0 & 0 \\
				0 & u_1 & 0 & 4\frac{\phi_4}{\phi_1} & -4q(\gamma-1)\frac{\phi_5}{\phi_1} \\
				0 & 0 & u_1 & 0 & 0 \\
				0 & 0 & 0 & \left[(2-\gamma) - \frac{q\rho Y (\gamma-1)^2}{p+q\rho Y(\gamma-1)}\right]u_1 & 2q u_1 (\gamma - 1)^2\frac{\phi_5}{\phi_4} \\
				0 & 0 & 0 & 0 & u_1
			\end{pmatrix}, \\
			B_1 = \frac{P^{-1}\widetilde{B}}{2}
			=\frac{1}{2}
			\begin{pmatrix}
				u_2 & 0 & 0 & 0 & 0 \\
				0 & u_2 & 0 & 0 & 0 \\
				0 & 0 & u_2 & 0 & 0 \\
				0 & 0 & 2(\gamma-1)\frac{\phi_4}{\phi_1} & \left[(2-\gamma) - \frac{q\rho Y (\gamma-1)^2}{p+q\rho Y(\gamma-1)}\right]u_2 & 0 \\
				0 & 0 & -2q(\gamma-1)^2 \frac{\phi_5}{\phi_1} & 2q u_2 (\gamma - 1)^2\frac{\phi_5}{\phi_4} & u_2
			\end{pmatrix}, \\
			B_2= \frac{P^{-1}\widetilde{B}^{T}}{2}
			=\frac{1}{2}
			\begin{pmatrix}
				u_2 & 0 & 0 & 0 & 0 \\
				0 & u_2 & 0 & 0 & 0 \\
				0 & 0 & u_2 & 4\frac{\phi_4}{\phi_1} & -4q(\gamma-1)\frac{\phi_5}{\phi_1} \\
				0 & 0 & 0 & \left[(2-\gamma) - \frac{q\rho Y (\gamma-1)^2}{p+q\rho Y(\gamma-1)}\right]u_2 & 2q u_2 (\gamma - 1)^2\frac{\phi_5}{\phi_4} \\
				0 & 0 & 0 & 0 & u_2
			\end{pmatrix}.
		\end{split}
	\end{equation}
	Note that the above matrices are independent of $\alpha$ and thus of the norm $P$.
	This completes the third step in Remark~\ref{rem:construction_procedure}.
	
	Having the above skew-symmetric formulation \eqref{eq:final_form0},  we define the quadratic energy
	\begin{equation} \label{eq:energy_norm}
		E(t) = \|\Phi\|_P^2 = \int_{\Omega} \Phi^T P \Phi \,\mathrm{d}\Omega = \int_\Omega
		\left[
		\alpha^2\rho
		+ (\gamma-1) \rho E
		+\rho Y
		\right]
		d\Omega
	\end{equation}
	and obtain the corresponding energy equation
	\begin{equation} \label{eq:energy_rate_final}
		\frac{d}{dt} \|\Phi\|_P^2
		+ \oint_{\partial \Omega} \Phi^{T} (n_{x_1}\tilde{A} + n_{x_2}\tilde{B} ) \Phi \,\mathrm{d}s
		=-\int_{\Omega} 2\Phi^T P C \Phi \,\mathrm{d}\Omega
		= \int_{\Omega} \omega \,\mathrm{d}\Omega \leq 0.
	\end{equation}
	where $(n_{x_1}, n_{x_2})$ denotes the outward unit normal vector on the boundary $\partial \Omega$.
	Hence the system is energy-stable under appropriate boundary conditions, \eg, periodic boundary conditions.
	
	\begin{rem}
		It is interesting to note that the energy $\Phi^T P \Phi = \alpha^2\rho + (\gamma-1) \rho E +\rho Y$ integrated in \eqref{eq:energy_norm}  is a linear combination of the conservative variables $\rho, \rho E$ and $\rho Y$. This seems natural as the energy should be conserved generally. This may motivate us to define the new energy function for general hyperbolic systems of conservation laws, \ie, choose $\Phi^T P \Phi$ as a linear combination of the corresponding conservative variables. In this way, the procedure to construct $\Phi$ and $P$ could be simplified.
	\end{rem}

	\subsection{Skew-symmetric formulation of the reactive Navier--Stokes equations}
	
	We further derive the skew-symmetric formulation for the 2D reactive Navier--Stokes equations. The same transformed variables $\phi_1, \phi_2, \ldots,\phi_5$ in \eqref{eq:new_vars} are employed.
	Following the procedure for the reactive Euler equations, we first give the reactive Navier--Stokes equations in primitive variables:
	
	\begin{equation}\label{eq:rns_primitive}
		\begin{aligned}
			\rho_t + u_1\rho_{x_1} + u_2\rho_{x_2} + \rho(u_{1x_1} + u_{2x_2}) = 0, \\[5pt]
			u_{1t} + u_1u_{1x_1} + u_2u_{1x_2} + \frac{p_{x_1}}{\rho}
			= \sum_{i=1}^{2}\frac{(\tau_{1i})_{x_i}}{\rho}
			&\coloneqq S_2, \\[5pt]
			u_{2t} + u_1u_{2x_1} + u_2u_{2x_2} + \frac{p_{x_2}}{\rho}
			= \sum_{i=1}^{2}\frac{(\tau_{2i})_{x_i}}{\rho}
			&\coloneqq S_3, \\[5pt]
			\bigl[p + (\gamma - 1)q\rho Y\bigr]_t
			+ \bigl[(p + (\gamma - 1)q\rho Y)u_1\bigr]_{x_1}\\
			+ \bigl[(p + (\gamma - 1)q\rho Y)u_2\bigr]_{x_2}
			+ (\gamma - 1)p(u_{1x_1} + u_{2x_2})\\
			= (\gamma -1)\sum_{i=1}^{2}\Bigl[
			(\kappa T_{x_i})_{x_i}
			+\sum_{j=1}^{2}\tau_{ij}(u_j)_{x_i}
			+ q\sigma (Y_{x_i})_{x_i}
			\Bigr]
			&\coloneqq S_4, \\[5pt]
			(\rho Y)_t + u_1(\rho Y)_{x_1} + u_2(\rho Y)_{x_2}
			+ \rho Y(u_{1x_1} + u_{2x_2}) - \omega
			= \sum_{i=1}^{2}(\sigma Y_{x_i})_{x_i}
			&\coloneqq S_5.
		\end{aligned}
	\end{equation}
	Then we collect the viscous and diffusive contributions in the above equations into the source vector
	$S = (0,S_2,S_3,S_4,S_5)^T$. From these equations we can obtain the equation for $\Phi$:
	\begin{equation}\label{eq:rns_compact}
		\begin{split}
			\Phi_t + A\Phi_{x_1} + B\Phi_{x_2} + C\Phi &= \Lambda S,
		\end{split}
	\end{equation}
	where $A,B,C$ are given in \eqref{eq:matrix_ABC} and
	$\Lambda = \text{diag}\left(\frac{1}{2\phi_1}, \phi_1, \phi_1, \frac{1}{2\phi_4}, \frac{1}{2\phi_5} \right)$.
	%
	According to the skew-symmetric forms of the reactive Euler equations \eqref{eq:final_form0} and \eqref{eq:final_form}, the above equation can be written as 	
	\begin{equation}\label{eq:final_form_ns0}
		P \Phi_t + (\frac{1}{2}\tilde A \Phi)_{x_1} + \frac{1}{2}\tilde A^T \Phi_{x_1}
		+ (\frac{1}{2}\tilde B \Phi)_{x_2} + \frac{1}{2}\tilde B^T  \Phi_{x_2} + PC\Phi = P \Lambda S,
	\end{equation}
	or equivalently,
	\begin{equation}\label{eq:final_form_ns}
		\Phi_t + (A_1\Phi)_{x_1} + A_2\Phi_{x_1} + (B_1\Phi)_{x_2} + B_2\Phi_{x_2} + C\Phi = \Lambda S,
	\end{equation}
	where $\tilde A$, $\tilde B$, and $A_1, A_2, B_1, B_2$ are given in \eqref{eq:ABtilde} and \eqref{eq:A1B1}, respectively.
	
	To show that \eqref{eq:final_form_ns0} satisfies the condition \eqref{eq:parabolic_divergence}, we compute
	\begin{equation}\label{eq:Phi_PLS}
		\begin{aligned}
			\Phi^{T}(P\Lambda S)
			=&\frac{1}{2}(\gamma-1)\Bigg(
			\sum_{i,j=1}^{2}u_i(\tau_{ij})_{x_j}
			+\sum_{i=1}^{2}(\kappa T_{x_i})_{x_i}
			+\sum_{i,j=1}^{2}\tau_{ij}(u_i)_{x_j}
			+\sum_{i=1}^{2}q\sigma(Y_{x_i})_{x_i}
			\Bigg) \\
			&+\frac{1}{2}\sum_{i=1}^{2}(\sigma Y_{x_i})_{x_i}\\
			=& \frac{1}{2} (\gamma-1)\sum_{j=1}^{2}
			\frac{\partial}{\partial x_j}
			\left(
			\sum_{i=1}^{2}u_i\tau_{ij}
			+\kappa T_{x_j}
			+(q+\frac{1}{\gamma-1})\sigma Y_{x_j}
			\right).
		\end{aligned}
	\end{equation}
	Note that the above result is in a divergence form, leading to the satisfaction of \eqref{eq:parabolic_divergence}.
	Thus \eqref{eq:final_form_ns0} is the skew-symmetric formulation of the 2D reactive Navier--Stokes equations.
	
	Multiplying \eqref{eq:final_form_ns0} by $\Phi^T$ from the left and integrating over $\Omega$, we obtain the energy equation
	\begin{equation}\label{eq:ns_energy_balance_continuous}
		\begin{aligned}
			&\frac{d}{dt}\|\Phi\|_P^2
			+ \oint_{\partial\Omega}
			\Phi^{T}(n_{x_1}\tilde{A}+n_{x_2}\tilde{B})\Phi \,\mathrm{d}s \\
			& \hspace{5mm} -(\gamma-1)\oint_{\partial\Omega}
			\Bigg[
			\sum_{i,j=1}^{2}u_i\tau_{ij}n_{x_j}
			+\kappa\sum_{j=1}^{2}T_{x_j}n_{x_j}
			+\left(q+\frac{1}{\gamma-1}\right)
			\sigma\sum_{j=1}^{2}Y_{x_j}n_{x_j}
			\Bigg] \,\mathrm{d}s \\
			&=
			-\int_{\Omega}2\Phi^{T}PC\Phi \,\mathrm{d}\Omega
			=\int_{\Omega} \omega \,\mathrm{d}\Omega \leq 0.
		\end{aligned}
	\end{equation}
	Under periodic boundary conditions, the system is energy stable.
	Here we emphasize that, under the present transformed norm, the viscous and diffusive terms contribute only through boundary fluxes and do not yield a negative-definite volume dissipation term.
	
	\section{Energy stable schemes for reactive Euler equations}
	\label{sec:es}
	\setcounter{equation}{0}
	In this section, we construct spatial discretizations for the skew-symmetric formulation of the reactive Euler equations by adding dissipation terms to the framework developed in \cite{Nordstrom2022PrimalDual}. The extension to the reactive Navier--Stokes equations is presented in Section~\ref{sec:ns_discrete}.
	Specifically, the first-order derivatives are approximated using SBP operators as in \cite{Nordstrom2022PrimalDual}. Based on the upwind SBP-WENO operator in \cite{JiangWang2024UpwindSBP}, we propose an efficient artificial dissipation operator to suppress spurious oscillations near discontinuities. In this way, we obtain a semi-discrete scheme that is not only energy stable but also
	robust for discontinuous problems. 
	
	\subsection{General framework}
	\label{subsec:euler_discrete}
	
	We first give the general framework and then provide the specific operators in the next subsection.
	
	\subsubsection{Semi-discrete scheme}
	
	We consider a 2D rectangular domain $\Omega$ and discretize it using a uniform Cartesian mesh with
	\(N_{x}\times N_{y}\) grid points. The corresponding mesh sizes along $x$ and $y$ directions are denoted as $\delta_x$ and $\delta_y$, respectively.
	%
	%
	Let $\Phi_{ij} \in \mathbb{R}^{5} $ be the numerical solution of $\Phi$ at point $(x_i,y_j)$. Then the general semi-discrete finite difference scheme for the skew-symmetric form \eqref{eq:final_form} can be written as
	\begin{equation}\label{eq:euler_pointwise}
		\begin{aligned}
			\frac{d}{dt}\Phi_{ij}
			&+
			\bigl(D_x(A_1\Phi)\bigr)_{ij}
			+
			A_2(x_i,y_j)\bigl(D_x\Phi\bigr)_{ij}
			\\
			&+
			\bigl(D_y(B_1\Phi)\bigr)_{ij}
			+
			B_2(x_i,y_j)\bigl(D_y\Phi\bigr)_{ij}+
			C(x_i,y_j)\Phi_{ij}
			\\
			&=
			-\bigl(D_{\mathrm{diss},x}\Phi\bigr)_{ij}
			-\bigl(D_{\mathrm{diss},y}\Phi\bigr)_{ij}.
		\end{aligned}
	\end{equation}
	Here $D_x$ and $D_y$ are linear finite difference operators such that
	$(D_x(A_1\Phi))_{ij}$ and $(D_y(B_1\Phi))_{ij}$ approximate the first-order derivatives
	$\partial_x (A_1\Phi)$ and $\partial_y (B_1\Phi)$ at $(x_i,y_j)$, respectively. The terms $\bigl(D_{\mathrm{diss},x}\Phi\bigr)_{ij}$ and $\bigl(D_{\mathrm{diss},y}\Phi\bigr)_{ij}$ are artificial dissipations, which serve only to stabilize the numerical scheme near discontinuities and become high-order small terms for smooth solutions.
	
	For each scalar component \(\phi_k\), we collect its numerical solutions $\phi_{k,ij}$ at all grid points $(x_i,y_j), i=1,2,...,N_x, j=1,2,...,N_y$ into a vector as
	\[
	\vec{\Phi}_k
	=
	\begin{pmatrix}
		\phi_{k,11}\\
		\phi_{k,21}\\
		\vdots\\
		\phi_{k,N_xN_y}
	\end{pmatrix}
	\in
	\mathbb{R}^{N_{x}N_{y}}.
	\]
	Define the full state vector
	\(
	\vec{\Phi}
	=
	\bigl(
	\vec{\Phi}_1^T,
	\vec{\Phi}_2^T,
	\ldots,
	\vec{\Phi}_5^T
	\bigr)^T
	\in
	\mathbb{R}^{5N_{x}N_{y}}.
	\)
	Then under periodic boundary conditions, the pointwise scheme \eqref{eq:euler_pointwise}
	can be assembled into the following form:
	\begin{equation}
		\label{eq:euler_semidiscrete}
		\vec{\Phi}_t
		+
		\mathbf{D}_{x}(\mathbf{A}_1\vec{\Phi})
		+
		\mathbf{A}_2\mathbf{D}_{x}\vec{\Phi}
		+
		\mathbf{D}_{y}(\mathbf{B}_1\vec{\Phi})
		+
		\mathbf{B}_2\mathbf{D}_{y}\vec{\Phi}
		+
		\mathbf{C}\vec{\Phi}
		=
		-\mathbf{D}_{\mathrm{diss}}\vec{\Phi},
	\end{equation}
	where \(\mathbf{A}_1,\mathbf{A}_2,\mathbf{B}_1,\mathbf{B}_2, \mathbf{C}\) are $5N_{x}N_{y} \times 5N_{x}N_{y}$ matrices defined in a similar way.
	For example, denote $A_1=(a_{ij})$, then
	\begin{equation}\label{eq:A1}
		\mathbf{A}_1
		=
		\begin{pmatrix}
			\mathbf a_{11} & \cdots & \mathbf a_{15}\\
			\vdots & \ddots & \vdots\\
			\mathbf a_{51} & \cdots & \mathbf a_{55}
		\end{pmatrix}		\in
		\mathbb{R}^{5N_{x}N_{y}
			\times
			5N_{x}N_{y}},
	\end{equation}
	where
	\(
	\mathbf a_{ij}
	=
	\operatorname{diag}
	\bigl(
	a_{ij}(x_{1},y_{1}),\ldots,
	a_{ij}(x_{N_x},y_{N_y})
	\bigr) \in
	\mathbb{R}^{N_xN_y\times N_xN_y}.
	\)
	%
	The two-dimensional derivative operators $\mathbf{D}_x$, $\mathbf{D}_y$
	and the artificial dissipation operator $\mathbf{D}_{\mathrm{diss}}$ are assembled from their one-dimensional
	counterparts via Kronecker product, the definition of which is given below.
	
	\begin{defn}[Kronecker product]\label{def:kron}
		Let $M \in \mathbb{R}^{m \times n}$ and $N \in \mathbb{R}^{p \times q}$.
		The Kronecker product $M \otimes N$ is the block matrix
		\[
		M \otimes N
		=
		\begin{pmatrix}
			m_{11} N & \cdots & m_{1n} N \\
			\vdots   & \ddots & \vdots   \\
			m_{m1} N & \cdots & m_{mn} N
		\end{pmatrix}
		\in \mathbb{R}^{mp \times nq}.
		\]
	\end{defn}
	
	With the above definition, the two-dimensional operators are
	\begin{equation}
		\label{eq:2D_operator}
		\mathbf{D}_{x}
		=
		I_5 \otimes D_{x} \otimes I_{N_y},
		\qquad
		\mathbf{D}_{y}
		=
		I_5 \otimes I_{N_x} \otimes D_{y},
	\end{equation}
	and
	\begin{equation}\label{eq:2d_diss}
		\mathbf{D}_{\mathrm{diss}}
		=
		\mathbf{D}_{\mathrm{diss},x} + \mathbf{D}_{\mathrm{diss},y}
	\end{equation}
	with
	\[
	\mathbf{D}_{\mathrm{diss},x}
	=
	I_5 \otimes D_{\mathrm{diss},x} \otimes I_{N_y},
	\qquad
	\mathbf{D}_{\mathrm{diss},y}
	=
	I_5 \otimes I_{N_x} \otimes D_{\mathrm{diss},y}.
	\]
	Here $I_{s}$ denotes the \(s\times s\) identity matrix.
	
	On the other hand, similar to the definition of $\mathbf{A}_1$ in \eqref{eq:A1}, we define $\tilde{\mathbf{A}}$  and $\mathbf{P}$ with each entry being a diagonal matrix determined by $\tilde A$ in \eqref{eq:ABtilde} and $P$ in \eqref{eq:P_sol}, respectively. Then the coefficient matrices in the scheme \eqref{eq:euler_semidiscrete} obey the following relation corresponding to \eqref{eq:A1B1}:
	\begin{equation}
		\label{eq:euler_compatibility}
		\mathbf{A}_1=\mathbf{P}^{-1}\tilde{\mathbf{A}}/2,\quad
		\mathbf{A}_2=\mathbf{P}^{-1}\tilde{\mathbf{A}}^{T}/2,\quad
		\mathbf{B}_1=\mathbf{P}^{-1}\tilde{\mathbf{B}}/2,\quad
		\mathbf{B}_2=\mathbf{P}^{-1}\tilde{\mathbf{B}}^{T}/2.
	\end{equation}
	Note that since $P$ is a constant diagonal matrix, $\mathbf{P}$ is also diagonal and can be written via the Kronecker product as
	\begin{equation}\label{eq:P_mathbf}
		\mathbf{P}=P\otimes I_{N_x}\otimes I_{N_y}.
	\end{equation}

	\subsubsection{Energy stability based on SBP operators}
	
	To preserve the energy stability of the continuous reactive Euler equations, we introduce summation-by-parts (SBP) operators as in \cite{Nordstrom2022PrimalDual}. These operators provide a discrete analogue of integration by parts and guarantee the discrete energy stability of the scheme \eqref{eq:euler_semidiscrete} under periodic boundary conditions.
	We begin by recalling the definition of one-dimensional SBP operators \cite{SvardNordstrom2014SBPReview,FernandezHickenZingg2014SBPSATReview,mattsson2017diagonal}.
	
	
	\begin{defn}[Diagonal-norm SBP operator]
		\label{def:sbp}
		Let $D$ be a consistent first-order derivative operator defined on the domain $[x_{\min}, x_{\max}]$ with grid points $x_1, x_2, \dots, x_N$ ($x_1=x_{\min}, x_N = x_{\max}$). Then $D$ is called a diagonal-norm SBP operator if it admits a decomposition $D = H^{-1} Q$, where $H=\diag(h_{11}, h_{22}, ..., h_{NN})$ is a diagonal, symmetric positive definite matrix (the norm matrix) satisfying $\sum_{i} h_{ii} = x_{\max} - x_{\min}$ and $Q$ satisfies the SBP constraint
		\begin{equation}\label{eq:sbp_def}
			Q + Q^T = B := \diag(-1,0,\ldots,0,1) \quad \mbox{for non-periodic problems},
		\end{equation}
		or
		\begin{equation}\label{eq:sbp_def0}
			Q + Q^T = B := 0 \quad \mbox{for periodic problems}.
		\end{equation}
		
	\end{defn}
	
	If the 1D operators $D_x$ and $D_y$ in \eqref{eq:2D_operator} are SBP operators, then the 2D operators $\mathbf{D}_{x}$ and $\mathbf{D}_{y}$ in \eqref{eq:2D_operator} satisfy the discrete integration-by-parts identities. To see this, we define
	\begin{equation}\label{eq:HB}
		\mathbf{H}
		=
		I_5\otimes H_{x}\otimes H_{y}, \quad
		\mathbf{B}_{x}
		=
		I_5\otimes B_{x}\otimes H_{y},
		\quad
		\mathbf{B}_{y}
		=
		I_5\otimes H_{x}\otimes B_{y},
	\end{equation}
	where $H_{x}, B_x$ and $H_y, B_y$ are the corresponding matrices ($H$ and $B$) in the definition of SBP operators $D_x$ and $D_y$, respectively. For arbitrary vectors $\vec U, \vec V \in \mathbb{R}^{5N_{x}N_{y}}$, a direct computation gives
	\begin{equation}
		\label{eq:sbp_vector_x1}
		\begin{aligned}
			\vec U^{T} \mathbf{H} \mathbf{D}_{x}\vec V
			&=
			\vec U^{T}
			(I_5\otimes H_x\otimes H_y)
			(I_5\otimes D_x\otimes I_{N_y})
			\vec V \\
			&=
			\vec U^{T}
			(I_5\otimes H_xD_x\otimes H_y)
			\vec V \\
			&=
			\vec U^{T}
			(I_5\otimes(-D_x^TH_x+B_x)\otimes H_y)
			\vec V \\
			&=
			-\vec U^{T}
			(I_5\otimes D_x^TH_x\otimes H_y)
			\vec V
			+
			\vec U^{T}
			(I_5\otimes B_x\otimes H_y)
			\vec V \\
			&=
			-(\mathbf{D}_{x}\vec U)^{T} \mathbf{H} \vec V
			+
			\vec U^{T}\mathbf{B}_{x}\vec V ,
		\end{aligned}
	\end{equation}
	where the SBP property $H_xD_x = -D_x^TH_x+B_x$ of the 1D operator $D_x$ has been used for the third equality.
	Similarly, we can also obtain
	\begin{equation}
		\label{eq:sbp_vector_x2}
		\vec U^{T} \mathbf{H} \mathbf{D}_{y}\vec V
		=
		-(\mathbf{D}_{y}\vec U)^{T} \mathbf{H} \vec V
		+
		\vec U^{T}\mathbf{B}_{y}\vec V .
	\end{equation}
	Equations \eqref{eq:sbp_vector_x1} and \eqref{eq:sbp_vector_x2} are the discrete analogues of continuous integration-by-parts identities and are crucial for the energy stability of the numerical schemes.
	
	
	
	Introduce the norm
	\[
	\|\vec U\|_{\mathbf P \mathbf H}^2
	=
	\vec U^T \mathbf P \mathbf H \vec U,
	\]
	where $\mathbf P$ and $\mathbf H$ are constant diagonal positive-definite matrices defined in \eqref{eq:P_mathbf} and \eqref{eq:HB}, respectively.
	Then we have
	\begin{thm}\label{thm:euler_energy_stability}
		For 2D periodic problems, suppose that the first-order derivative operators \(\mathbf{D}_{x}\) and \(\mathbf{D}_{y}\) in the scheme \eqref{eq:euler_semidiscrete}
		satisfy the SBP
		identities \eqref{eq:sbp_vector_x1} and \eqref{eq:sbp_vector_x2}, respectively,
		and that the artificial dissipation operator \(\mathbf{D}_{\mathrm{diss}}\)
		satisfies
		\begin{equation}\label{eq:Ddiss_cond}
			\vec{U}^T{\mathbf{H}}\mathbf{D}_{\mathrm{diss}}\vec{U}\ge 0, \quad \forall \vec{U}.
		\end{equation}
		Then the semi-discrete scheme \eqref{eq:euler_semidiscrete} satisfies the energy stability
		\begin{equation}
			\label{eq:euler_energy_diss_general}
			\frac{\mathrm{d}}{\mathrm{d}t}
			\|\vec{\Phi}\|_{\mathbf{P}{\mathbf{H}}}^2
			\le 0.
		\end{equation}
	\end{thm}
	\begin{proof}
		Multiplying \eqref{eq:euler_semidiscrete} from the left by
		\(2\vec{\Phi}^T\mathbf{P}{\mathbf{H}}\) and using
		\(2\vec{\Phi}^T(\mathbf{P}{\mathbf{H}})\vec{\Phi}_t
		= \frac{\mathrm{d}}{\mathrm{d}t}\|\vec{\Phi}\|_{\mathbf{P}{\mathbf{H}}}^2\),
		we obtain
		\[
		\begin{aligned}
			\frac{\mathrm{d}}{\mathrm{d}t}\|\vec{\Phi}\|_{\mathbf{P}{\mathbf{H}}}^2
			&+
			2\vec{\Phi}^{T}{\mathbf{H}}\mathbf{D}_{x}
			\bigl(\mathbf{P}\mathbf{A}_1\vec{\Phi}\bigr)
			+
			2\vec{\Phi}^{T}{\mathbf{H}}
			(\mathbf{P}\mathbf{A}_2)\mathbf{D}_{x}\vec{\Phi}
			\\
			&+
			2\vec{\Phi}^{T}{\mathbf{H}}\mathbf{D}_{y}
			\bigl(\mathbf{P}\mathbf{B}_1\vec{\Phi}\bigr)
			+
			2\vec{\Phi}^{T}{\mathbf{H}}
			(\mathbf{P}\mathbf{B}_2)\mathbf{D}_{y}\vec{\Phi}
			\\
			&=
			-2\vec{\Phi}^{T}\mathbf{P}{\mathbf{H}}
			\bigl(\mathbf{C}+\mathbf{D}_{\mathrm{diss}}\bigr)\vec{\Phi}.
		\end{aligned}
		\]
		From relations \eqref{eq:euler_compatibility}, we have
		\(2\mathbf{P}\mathbf{A}_1 = \tilde{\mathbf{A}}\),
		\(2\mathbf{P}\mathbf{A}_2 = \tilde{\mathbf{A}}^{T}\),
		\(2\mathbf{P}\mathbf{B}_1 = \tilde{\mathbf{B}}\),
		\(2\mathbf{P}\mathbf{B}_2 = \tilde{\mathbf{B}}^{T}\).
		Substituting these into the above equation yields
		\begin{equation}\label{eq:thm41_1}
			\begin{aligned}
				\frac{\mathrm{d}}{\mathrm{d}t}\|\vec{\Phi}\|_{\mathbf{P}{\mathbf{H}}}^2
				&+
				\vec{\Phi}^{T}{\mathbf{H}}\mathbf{D}_{x}
				\bigl(\tilde{\mathbf{A}}\vec{\Phi}\bigr)
				+
				\vec{\Phi}^{T}{\mathbf{H}}\tilde{\mathbf{A}}^{T}
				\mathbf{D}_{x}\vec{\Phi}
				\\
				&+
				\vec{\Phi}^{T}{\mathbf{H}}\mathbf{D}_{y}
				\bigl(\tilde{\mathbf{B}}\vec{\Phi}\bigr)
				+
				\vec{\Phi}^{T}{\mathbf{H}}\tilde{\mathbf{B}}^{T}
				\mathbf{D}_{y}\vec{\Phi}
				\\
				&=
				-2\vec{\Phi}^{T}\mathbf{P}{\mathbf{H}}
				\bigl(\mathbf{C}+\mathbf{D}_{\mathrm{diss}}\bigr)\vec{\Phi}.
			\end{aligned}
		\end{equation}
		We now apply the 2D SBP identities
		\eqref{eq:sbp_vector_x1}--\eqref{eq:sbp_vector_x2}.
		Taking \(\vec{U}=\vec{\Phi}\) and \(\vec{V}=\tilde{\mathbf{A}}\vec{\Phi}\)
		in \eqref{eq:sbp_vector_x1} gives
		\[
		\vec{\Phi}^{T}{\mathbf{H}}\mathbf{D}_{x}
		\bigl(\tilde{\mathbf{A}}\vec{\Phi}\bigr)
		+
		\vec{\Phi}^{T}{\mathbf{H}}\tilde{\mathbf{A}}^{T}
		\mathbf{D}_{x}\vec{\Phi}
		=
		\vec{\Phi}^{T}\mathbf{B}_{x}\tilde{\mathbf{A}}\vec{\Phi}.
		\]
		Similarly, taking \(\vec{V}=\tilde{\mathbf{B}}\vec{\Phi}\)
		in \eqref{eq:sbp_vector_x2} yields
		\[
		\vec{\Phi}^{T}{\mathbf{H}}\mathbf{D}_{y}
		\bigl(\tilde{\mathbf{B}}\vec{\Phi}\bigr)
		+
		\vec{\Phi}^{T}{\mathbf{H}}\tilde{\mathbf{B}}^{T}
		\mathbf{D}_{y}\vec{\Phi}
		=
		\vec{\Phi}^{T}\mathbf{B}_{y}\tilde{\mathbf{B}}\vec{\Phi}.
		\]
		Under periodic boundary conditions, $\mathbf{B}_{x} = \mathbf{B}_{y}=0$ since $B_x=B_y=0$ according to \eqref{eq:sbp_def0}. Thus the above two terms vanish and \eqref{eq:thm41_1} becomes
		\[
		\frac{\mathrm{d}}{\mathrm{d}t}\|\vec{\Phi}\|_{\mathbf{P}{\mathbf{H}}}^2
		=
		-2\vec{\Phi}^{T}\mathbf{P}{\mathbf{H}}
		\bigl(\mathbf{C}+\mathbf{D}_{\mathrm{diss}}\bigr)\vec{\Phi}.
		\]
		Note that the matrix \(\mathbf{C}\) constructed from $C$ in \eqref{eq:matrix_ABC} is diagonal and positive semidefinite, thus $\mathbf{P}{\mathbf{H}} \mathbf{C}$ is positive semidefinite and $\vec{\Phi}^{T} \mathbf{P}{\mathbf{H}} \mathbf{C} \vec{\Phi} \ge 0$.
		
		On the other hand, we define
		\[
		\mathbf{P}^{1/2} = P^{1/2} \otimes I_{N_x} \otimes I_{N_y}
		\]
		and show that $\mathbf{P}^{1/2}$ commutes with both $\mathbf{H}$ and $\mathbf{D}_{\mathrm{diss}}$.
		The commutativity of $\mathbf{P}^{1/2}$ and $\mathbf{H}$ is obvious as they are diagonal matrices.
		Using the Kronecker product identity $(A \otimes B \otimes C)(D \otimes E \otimes F) = AD \otimes BE \otimes CF$,
		we can verify that $\mathbf{P}^{1/2}$ commutes with $\mathbf{D}_{\rm diss,x}$:
		\[
		\begin{aligned}
			\mathbf{P}^{1/2}\mathbf{D}_{\rm diss,x} &= \big(P^{1/2} \otimes I_{N_x} \otimes I_{N_y}\big)\big(I_5 \otimes D_{\rm diss,x} \otimes I_{N_y}\big) = P^{1/2} \otimes D_{\rm diss,x} \otimes I_{N_y}, \\
			\mathbf{D}_{\rm diss,x}\mathbf{P}^{1/2} &= \big(I_5 \otimes D_{\rm diss,x} \otimes I_{N_y}\big)\big(P^{1/2} \otimes I_{N_x} \otimes I_{N_y}\big) = P^{1/2} \otimes D_{\rm diss,x} \otimes I_{N_y},
		\end{aligned}
		\]
		Similarly, $\mathbf{P}^{1/2}$ commutes with $\mathbf{D}_{\rm diss,y}$, and thus further commutes with $\mathbf{D}_{\mathrm{diss}}$. With these, we have
		\[
		\begin{aligned}
			\vec{\Phi}^T \mathbf{P} \mathbf{H} \mathbf{D}_{\rm diss} \vec{\Phi}
			&= \vec{\Phi}^T \mathbf{P}^{1/2} \mathbf{H} \mathbf{D}_{\rm diss} \mathbf{P}^{1/2} \vec{\Phi} \\
			&= \big( \mathbf{P}^{1/2} \vec{\Phi} \big)^T \mathbf{H} \mathbf{D}_{\rm diss} \big( \mathbf{P}^{1/2} \vec{\Phi} \big) \ge 0.
		\end{aligned}
		\]
		Consequently, $\vec{\Phi}^{T}\mathbf{P}{\mathbf{H}}(\mathbf{C}+\mathbf{D}_{\mathrm{diss}})\vec{\Phi} \ge 0 $ and
		$
		\frac{\mathrm{d}}{\mathrm{d}t}
		\|\vec{\Phi}\|_{\mathbf{P}{\mathbf{H}}}^2 \le  0.
		$
		The proof is complete.
	\end{proof}

	\begin{rem}
		The above idea of constructing energy stable schemes based on the skew-symmetric form is mainly from \cite{Nordstrom2022PrimalDual,Nordstrom2022SkewEuler}, except for the dissipation operator $\mathbf{D}_{\mathrm{diss}}$.
		As will be seen in the numerical experiments, without $\mathbf{D}_{\mathrm{diss}}$ the numerical results suffer from serious oscillations. It is a challenging problem to design $\mathbf{D}_{\mathrm{diss}}$ that (1) satisfies \eqref{eq:Ddiss_cond} so that the scheme is energy stable, (2) effectively suppresses spurious oscillations near discontinuities, and (3) remains a high-order small term that does not affect the original convergence order for smooth solutions too much. We will propose such a dissipation operator in the next subsection.
	\end{rem}

	\subsection{Construction of operators}\label{sec:subsection_operator}
	
	In this subsection we construct SBP derivative operators and dissipation operators that satisfy the conditions of Theorem~\ref{thm:euler_energy_stability}.
	
	\subsubsection{SBP derivative operators}
	\label{subsec:sbp_construction}
	
	According to the above framework, once the 1D operators $D_x$ and $D_y$ in \eqref{eq:2D_operator} are SBP operators, the 2D operators $\mathbf{D}_{x}$ and $\mathbf{D}_{y}$ in \eqref{eq:2D_operator} satisfy the discrete integration-by-parts identities \eqref{eq:sbp_vector_x1} and \eqref{eq:sbp_vector_x2}, respectively.
	Here we construct the 1D SBP operators under periodic boundary conditions, while the corresponding treatment of non-periodic boundary conditions is left for future work.
	Both \(D_{x}\) and \(D_{y}\) are taken as the 1D SBP operator \(D\) defined below.
	
	Let $u$ be a scalar variable on the periodic domain $[x_{\min}, x_{\max}]$. Denote $\mathbf u = (u_1, u_2, \ldots, u_N)^T$ as the vector consisting of numerical solutions of $u$ at grid points $x_1, x_2, \ldots, x_N$ under a uniform mesh size $h$ ($x_1=x_{\min}, x_N = x_{\max}$). We define a pair of third-order upwind difference operators \cite{mattsson2017diagonal,JiangWang2024UpwindSBP}
	\begin{equation}\label{eq:D+D-}
		\begin{split}
			&D^-u|_i
			:=
			\frac1h
			\left(
			\frac16u_{i-2}
			-u_{i-1}
			+\frac12u_i
			+\frac13u_{i+1}
			\right), \\
			& D^+ u|_i :=
			\frac1h \left( -\frac{1}{3} u_{i-1} - \frac{1}{2} u_i + u_{i+1} - \frac{1}{6} u_{i+2} \right),
		\end{split}
	\end{equation}
	where $i=1,2,\ldots,N$.  For periodic problems, \(D^-\) and \(D^+\) are \(N\times N\) circulant matrices:
	\[
	D^- = \frac{1}{h}
	\begin{pmatrix}
		\frac12 & \frac13 & 0 & \cdots & 0 & \frac16 & -1 \\[4pt]
		-1 & \frac12 & \frac13 & 0 & \cdots & 0 & \frac16 \\[4pt]
		\frac16 & -1 & \frac12 & \frac13 & 0 & \cdots & 0 \\[4pt]
		\vdots & \ddots & \ddots & \ddots & \ddots & \ddots & \vdots \\[4pt]
		0 & \cdots & 0 & \frac16 & -1 & \frac12 & \frac13 \\[4pt]
		\frac13 & 0 & \cdots & 0 & \frac16 & -1 & \frac12
	\end{pmatrix},
	D^+ = \frac{1}{h}
	\begin{pmatrix}
		-\frac12 & 1 & -\frac16 & \cdots & 0 & 0 & -\frac13 \\[4pt]
		-\frac13 & -\frac12 & 1 & -\frac16 & \cdots & 0 & 0 \\[4pt]
		0 & -\frac13 & -\frac12 & 1 & \ddots & \cdots & 0 \\[4pt]
		\vdots & \ddots & \ddots & \ddots & \ddots & \ddots & \vdots \\[4pt]
		0 & \cdots & \cdots & -\frac13 & -\frac12 & 1 & -\frac16 \\[4pt]
		-\frac16 & 0 & \cdots & 0 & -\frac13 & -\frac12 & 1 \\[4pt]
		1 & -\frac16 & 0 & \cdots & 0 & -\frac13 & -\frac12
	\end{pmatrix}.
	\]
	%
	%
	It is direct to verify that
	\begin{equation}\label{eq:D_DT}
		D^+ + (D^-)^T=0.
	\end{equation}
	%
	With the above two upwind operators, we define the SBP operator as
	\begin{equation}
		\label{eq:Dskew_1d}
		D
		=
		\frac12(D^-+D^+),
	\end{equation}
	which corresponds to the fourth-order central difference scheme
	\begin{equation}\label{eq:scheme_without_stabilization}
		D u|_i
		=
		\frac1h
		\left(
		\frac1{12}u_{i-2}
		-\frac23u_{i-1}
		+\frac23u_{i+1}
		-\frac1{12}u_{i+2}
		\right).
	\end{equation}
	
	Define $H=h I_N$ with $I_N$ the \(N\times N\) identity matrix and set $Q=HD$.
	Using \eqref{eq:D_DT} and noting that $H$ is a scalar matrix, we have
	\[
	\begin{aligned}
		Q + Q^T
		&=
		HD +(HD)^T\\
		&=
		\frac12\left[
		HD^-+(HD^+)^T
		+
		HD^+ +(HD^-)^T
		\right]\\
		&= \frac12\left[
		H (D^- + (D^+)^T)
		+
		H (D^+ + (D^-)^T)
		\right]
		=0.
	\end{aligned}
	\]
	Thus according to Definition \ref{def:sbp}, $D$ is a SBP operator for periodic problems.
	
	\subsubsection{Artificial dissipation operator}
	\label{subsec:diss_construction}
	
	Next we construct an artificial dissipation operator
	\(\mathbf{D}_{\mathrm{diss}}\) that satisfies the dissipativity condition in Theorem~\ref{thm:Ddiss_psd}, \ie, $\vec{U}^T{\mathbf{H}}\mathbf{D}_{\mathrm{diss}}\vec{U}\ge 0, \forall \vec{U}$.
	According to \eqref{eq:2d_diss} and the definition
	\({\mathbf H} = I_5 \otimes H_x \otimes H_y\), it suffices to set both $D_{\mathrm{diss},x}$ and $D_{\mathrm{diss},y}$ as the 1D dissipation operator $D_{\mathrm{diss}}$ satisfying \(\mathbf{v}^{T}H D_{\mathrm{diss}} \mathbf{v} \ge 0\) for all \(\mathbf{v}\); see the proof of Theorem~\ref{thm:Ddiss_psd} below.
	%
	%
	Here we construct such an operator based on the biased one-sided upwind operator $D_{ws}^-$  from \cite{JiangWang2024UpwindSBP}, which incorporates nonlinear WENO weights and a stabilization technique.
	With $D_{ws}^-$, we develop its counterpart $D_{ws}^+$, the upwind operator biased towards the opposite side, and define  $D_{\mathrm{diss}}$ as
	\begin{equation}\label{eq:Ddiss_1d}
		D_{\mathrm{diss}}=\frac12(D_{ws}^- - D_{ws}^+).
	\end{equation}
	The details are given as follows.
	
	\paragraph{"Minus" upwind operator \(D_{ws}^-\).}
	Following \cite{JiangWang2024UpwindSBP} we rewrite the third-order upwind operator $D^-$ in \eqref{eq:D+D-} as
	\[
	D^-u|_i
	=
	\frac1h(\widehat u_i-\widehat u_{i-1}),\quad
	\widehat u_i
	=
	-\frac16u_{i-1}
	+\frac56u_i
	+\frac13u_{i+1}.
	\]
	Note that $\widehat u_i$ can be written as
	\begin{equation}\label{eq:hatui}
		\widehat u_i
		=
		d^{(1)}\widehat u_i^{(1),-}
		+
		d^{(2)}\widehat u_i^{(2),-},
	\end{equation}
	where
	\[
	\widehat u_i^{(1),-}
	=
	\frac12u_i+\frac12u_{i+1},
	\qquad
	\widehat u_i^{(2),-}
	=
	-\frac12u_{i-1}+\frac32u_i,
	\]
	and
	\[
	d^{(1)}=\frac23,\qquad d^{(2)}=\frac13.
	\]

	Next we formulate a WENO version of $\widehat u_i$ in \eqref{eq:hatui} by replacing the linear weights $d^{(1)}$ and $d^{(2)}$ with nonlinear ones. Define the WENO-Z type smoothness
	indicators \cite{YamaleevCarpenter2009ThirdOrderEnergyStableWENO}
	\[
	\beta_i^{(1),-}=(u_{i+1}-u_i)^2,
	\qquad
	\beta_i^{(2),-}=(u_i-u_{i-1})^2,
	\qquad
	\tau_i^-=(u_{i+1}-2u_i+u_{i-1})^2,
	\]
	and introduce the nonlinear weights
	\begin{equation}\label{eq:weno_weight}
		w_i^{(r),-}
		=
		\frac{\alpha_i^{(r),-}}
		{\alpha_i^{(1),-}+\alpha_i^{(2),-}}, \qquad
		\alpha_i^{(r),-}
		=
		d^{(r)}
		\left(
		1+\frac{\tau_i^-}{\varepsilon+\beta_i^{(r),-}}
		\right),
		\qquad r=1,2,
	\end{equation}
	where $\varepsilon >0$ is a small parameter. Then the WENO version of $\widehat u_i$ is
	\begin{equation}
		\label{eq:weno_flux_def}
		\widehat u_i^{\,w,-}
		=
		w_i^{(1),-}\widehat u_i^{(1),-}
		+
		w_i^{(2),-}\widehat u_i^{(2),-}
	\end{equation}
	and the corresponding raw WENO-type "minus" operator is \cite{JiangWang2024UpwindSBP}
	\begin{equation}
		\label{eq:Dw_minus_pointwise}
		D_w^-u|_i
		=
		\frac1h
		\left(\widehat u_i^{\,w,-}-\widehat u_{i-1}^{\,w,-}\right).
	\end{equation}
	%
	%
	Note that $D_w^-u|_i$ can be written as
	\begin{equation} \label{eq:Dw_minus_component}
		\begin{split}
			D_w^-u|_i
			&= \frac1h  \left(\widehat u_i^{\,w,-}-\widehat u_{i-1}^{\,w,-}\right)\\
			& = \frac1h \left(
			w_i^{(1),-}\widehat u_i^{(1),-}
			+
			w_i^{(2),-}\widehat u_i^{(2),-}
			- w_{i-1}^{(1),-}\widehat u_{i-1}^{(1),-}
			-
			w_{i-1}^{(2),-}\widehat u_{i-1}^{(2),-}
			\right) \\
			& = \frac1h \left[
			\frac12 w_{i-1}^{(2),-} u_{i-2}
			-\left( \frac12 w_{i}^{(2),-}  + \frac12 w_{i-1}^{(1),-} + \frac32 w_{i-1}^{(2),-} \right) u_{i-1} \right.\\
			& \hspace{12mm}
			\left.
			+\left( \frac12 w_{i}^{(1),-}  + \frac32 w_{i}^{(2),-} - \frac12 w_{i-1}^{(1),-} \right) u_{i}
			+ \frac12 w_{i}^{(1),-} u_{i+1}
			\right].
		\end{split}
	\end{equation}
	Under periodic boundary conditions, the above equations at all grid points $x_1, x_2, \ldots, x_N$ can be written in the vector form $H D_w^- \mathbf u$, where $H=h I_N$ and $D_w^-$ is an $N \times N$ matrix depending on the nonlinear weights.
	
	Following \cite{JiangWang2024UpwindSBP}, we decompose $D_w^-$ as
	\begin{equation}
		D_w^-
		=
		H^{-1}(Q_{w}^-+R_{w}^-),
	\end{equation}
	where
	\[
	Q_{w}^-
	=
	\frac12\left(HD_w^- - (HD_w^-)^T\right),
	\qquad
	R_{w}^-
	=
	\frac12\left(HD_w^-+(HD_w^-)^T\right)
	\]
	are skew-symmetric and symmetric matrices, respectively.
	Note that
	$$
	\mathbf v^T Q_{w}^- \mathbf v = \frac12 \mathbf v^T ( Q_{w}^- + (Q_{w}^-)^T ) \mathbf v = 0,
	\quad \forall \mathbf v \in \mathbb{R}^N,
	$$
	since $Q_{w}^-$ is skew-symmetric. Although $R_{w}^-$ is symmetric, it may not be positive semidefinite.
	In the following, we add a symmetric stabilization term \(R_s^-\) to obtain the final operator $D_{ws}^-$:
	\begin{equation}
		\label{eq:Dws_minus_def}
		D_{ws}^-
		:=D_w^- + H^{-1} R_s^-
		=
		H^{-1}\left(Q_{w}^- + R_{w}^-+R_s^-\right),
	\end{equation}
	so that $R_{w}^-+R_s^-$ is symmetric positive semidefinite.
	
	From the expression of $D_w^-$ in \eqref{eq:Dw_minus_component}, it is straightforward to verify that the entries of the symmetric matrix $R_{w}^-$ are given by
	\[
	(R_w^-)_{ij}
	=
	\begin{cases}
		\dfrac14 w_{i-1}^{(2),-},
		& j=i-2,\\[4pt]
		-\dfrac14\left(3w_{i-1}^{(2),-}+w_i^{(2),-}\right),
		& j=i-1,\\[4pt]
		w_i^{(2),-}+\dfrac12w_{i-1}^{(2),-},
		& j=i,\\[4pt]
		-\dfrac14\left(3w_i^{(2),-}+w_{i+1}^{(2),-}\right),
		& j=i+1,\\[4pt]
		\dfrac14 w_{i+1}^{(2),-},
		& j=i+2,\\[4pt]
		0,
		& \text{otherwise}.
	\end{cases}
	\]
	Then for any $\mathbf v = (v_1, v_2,..., v_N)^T$, the corresponding quadratic form can be written as
	\begin{equation*}
		\mathbf v^T R_w^- \mathbf v
		=  \sum_{i=1}^N \lambda_i^{-}(v_i-v_{i+1})^2
		+
		\sum_{i=1}^N \gamma_i^{-}
		(v_{i-1}-2v_i+v_{i+1})^2 ,
	\end{equation*}
	where the indices are periodic and
	\[
	\lambda_i^{-}
	=
	\frac14(w_i^{(2),-} - w_{i+1}^{(2),-}),
	\qquad
	\gamma_i^{-}
	=
	\frac14w_i^{(2),-}.
	\]
	From this, $R_w^-$ can be decomposed as
	\begin{equation}\label{eq:Rw_minus}
		R_{w}^-
		=
		\Delta_1^T\Lambda^{-}\Delta_1
		+
		\Delta_2^T\Gamma^{-}\Delta_2,
	\end{equation}
	where $\Lambda^{-} = \diag(\lambda_1^{-}, \lambda_2^{-}, \ldots, \lambda_N^{-}),
	\Gamma^{-} = \diag(\gamma_1^{-}, \gamma_2^{-}, \ldots, \gamma_N^{-})$ and $\Delta_1, \Delta_2$ are constant matrices defined by
	\[
	\Delta_1u|_i=u_i-u_{i+1},
	\qquad
	\Delta_2u|_i=u_{i-1}-2u_i+u_{i+1}.
	\]
	Note that the coefficient $\gamma_i^-$ is nonnegative and thus the second term in \eqref{eq:Rw_minus} is already positive semidefinite. In contrast, the coefficient \(\lambda_i^{-}\) may change sign due to differences of neighboring WENO weights. To fix this problem, we denote $\delta_1=h^3$	and define
	\[
	\lambda_{s,i}^{-}
	=
	\frac12
	\left(
	\sqrt{\left(\lambda_i^{-}\right)^2+\delta_1^2}
	-
	\lambda_i^{-}
	\right).
	\]
	Then
	\[
	\lambda_i^{-}+\lambda_{s,i}^{-}
	=
	\frac12
	\left(
	\lambda_i^{-}
	+
	\sqrt{\left(\lambda_i^{-}\right)^2+\delta_1^2}
	\right)
	\ge 0.
	\]
	From this we define
	\begin{equation}\label{eq:Rs}
		R_s^-
		=
		\Delta_1^T\Lambda_s^{-}\Delta_1,
		\qquad
		\Lambda_s^{-} = \diag(\lambda_{s,1}^{-}, \lambda_{s,2}^{-}, \ldots, \lambda_{s,N}^{-})
	\end{equation}
	and have
	\begin{equation}
		\label{eq:M_minus_psd}
		R_w^-+R_s^-
		=
		\Delta_1^T(\Lambda^{-}+\Lambda_s^{-})\Delta_1
		+
		\Delta_2^T\Gamma^{-}\Delta_2
		\ge 0.
	\end{equation}
	Therefore, we have determined the symmetric stabilization term \(R_s^-\) in \eqref{eq:Dws_minus_def}, the component form of which is given by
	\[
	R_s^- u|_i
	= \lambda_{s,i}^{-}(u_i-u_{i+1})
	-
	\lambda_{s,i-1}^{-}(u_{i-1}-u_i).
	\]
	Finally, the component form of $D_{ws}^-$ is
	\begin{equation}\label{eq:Dws_minus_pointwise}
		\begin{split}
			D_{ws}^-u|_i
			& = D_{w}^-u|_i + \frac1h R_s^- u|_i \\
			& =
			\frac{\widehat u_i^{\,w,-}-\widehat u_{i-1}^{\,w,-}}{h}
			+
			\frac1h
			\left[
			\lambda_{s,i}^{-}(u_i-u_{i+1})
			-
			\lambda_{s,i-1}^{-}(u_{i-1}-u_i)
			\right],
		\end{split}
	\end{equation}
	where $\widehat u_i^{\,w,-}$ is defined in \eqref{eq:weno_flux_def} with WENO weights.

	\paragraph{"Plus" upwind operator \(D_{ws}^+\).}
	
	Based on the construction of the biased one-sided upwind operator $D_{ws}^-$, we develop its counterpart $D_{ws}^+$, the upwind operator biased towards the opposite side.
	%
	%
	Similar to the definition of  $\widehat u_i$ in \eqref{eq:weno_flux_def}, we define $\widehat u_i^{\,w,+}$ as the WENO version of
	$D^+|u_i$ in \eqref{eq:D+D-}:
	\[
	\widehat u_i^{\,w,+}
	=
	w_i^{(1),+}\widehat u_i^{(1),+}
	+
	w_i^{(2),+}\widehat u_i^{(2),+},
	\]
	where
	\[
	\widehat u_i^{(1),+}
	=
	\frac12u_i+\frac12u_{i+1},
	\qquad
	\widehat u_i^{(2),+}
	=
	-\frac12u_{i+2}+\frac32u_{i+1}
	\]
	and the WENO weights are given by
	\[
	w_i^{(r),+}
	=
	\frac{\alpha_i^{(r),+}}
	{\alpha_i^{(1),+}+\alpha_i^{(2),+}},
	\qquad
	\alpha_i^{(r),+}
	=
	d^{(r)}
	\left(
	1+\frac{\tau_i^+}{\varepsilon+\beta_i^{(r),+}}
	\right),	
	\qquad r=1,2.
	\]
	In the above equation, the smoothness indicators are
	\[
	\beta_i^{(1),+}=(u_{i+1}-u_i)^2,
	\qquad
	\beta_i^{(2),+}=(u_{i+2}-u_{i+1})^2,
	\qquad
	\tau_i^+=(u_i-2u_{i+1}+u_{i+2})^2.
	\]
	With these, we define the raw WENO-type "plus" operator
	\begin{equation}
		\label{eq:Dw_plus_pointwise}
		D_w^+u|_i
		=
		\frac1h
		\left(\widehat u_i^{\,w,+}-\widehat u_{i-1}^{\,w,+}\right),
	\end{equation}
	which can be written as
	\begin{equation} \label{eq:Dw_plus_component}
		\begin{split}
			D_w^+u|_i
			&= \frac1h  \left(\widehat u_i^{\,w,+}-\widehat u_{i-1}^{\,w,+}\right)\\
			& = \frac1h \left(
			w_i^{(1),+}\widehat u_i^{(1),+}
			+
			w_i^{(2),+}\widehat u_i^{(2),+}
			- w_{i-1}^{(1),+}\widehat u_{i-1}^{(1),+}
			-
			w_{i-1}^{(2),+}\widehat u_{i-1}^{(2),+}
			\right) \\
			& = \frac1h \left[
			-\frac12 w_{i-1}^{(1),+} u_{i-1}
			+\left( \frac12 w_i^{(1),+} - \frac12 w_{i-1}^{(1),+}
			-\frac32 w_{i-1}^{(2),+} \right)u_i \right.\\
			& \hspace{12mm}
			\left.
			+\left( \frac12 w_i^{(1),+}+\frac32 w_i^{(2),+}
			+\frac12 w_{i-1}^{(2),+} \right)u_{i+1}
			-\frac12 w_i^{(2),+} u_{i+2}
			\right].
		\end{split}
	\end{equation}
	Under periodic boundary conditions, the above equations at all grid points $x_1, x_2, \ldots, x_N$ can be written in the vector form $H D_w^+ \mathbf u$, where $H=h I_N$ and $D_w^+$ is an $N \times N$ matrix depending on the nonlinear weights.
	
	With the above $D_w^+$, we define the "plus" operator $D_{ws}^+$ as
	\begin{equation}
		\label{eq:Dws_plus_def}
		D_{ws}^+
		:=D_w^+ + H^{-1} R_s^+
		=
		H^{-1}\left(Q_{w}^+ + R_{w}^+ +R_s^+\right),
	\end{equation}
	where
	\[
	Q_w^+
	=
	\frac12\left(HD_w^+ - (HD_w^+)^T\right),
	\qquad
	R_w^+
	=
	\frac12\left(HD_w^+ + (HD_w^+)^T\right),
	\]
	and $R_s^+$ is a stabilization term so that $R_w^+ + R_s^+$ is symmetric negative semidefinite.
	To determine $R_s^+$, we denote
	\[
	\lambda_i^{+}
	=
	\frac14(w_i^{(2),+} - w_{i-1}^{(2),+}),
	\qquad
	\gamma_i^{+}
	=
	\frac14w_i^{(2),+}
	\]
	and decompose $R_w^+$ as
	\begin{equation}\label{eq:Rw_plus}
		R_w^+
		=
		-\Delta_1^T\Lambda^{+}\Delta_1
		-
		(\Delta_3)^T\Gamma^{+}\Delta_3,
	\end{equation}
	where $\Lambda^{+} = \diag(\lambda_1^{+}, \lambda_2^{+}, \ldots, \lambda_N^{+})$,
	$\Gamma^{+} = \diag(\gamma_1^{+}, \gamma_2^{+}, \ldots, \gamma_N^{+})$ and $\Delta_1, \Delta_3$ are constant matrices defined by
	\[
	\Delta_1u|_i=u_i-u_{i+1},
	\qquad
	\Delta_3u|_i=u_i-2u_{i+1}+u_{i+2}.
	\]
	Having the decomposition \eqref{eq:Rw_plus}, we define
	\begin{equation}\label{eq:Rs+}
		R_s^+
		=
		- \Delta_1^T\Lambda_s^{+}\Delta_1,
		\qquad
		\Lambda_s^{+} = \diag(\lambda_{s,1}^{+}, \lambda_{s,2}^{+}, \ldots, \lambda_{s,N}^{+})
	\end{equation}
	with $\delta_1=h^3$ and
	\[
	\lambda_{s,i}^{+}
	=
	\frac12
	\left(
	\sqrt{\left(\lambda_i^{+}\right)^2+\delta_1^2}
	-
	\lambda_i^{+}
	\right).
	\]
	Then
	\[
	\lambda_i^{+}+\lambda_{s,i}^{+}
	=
	\frac12
	\left(
	\lambda_i^{+}
	+
	\sqrt{\left(\lambda_i^{+}\right)^2+\delta_1^2}
	\right)
	\ge 0
	\]
	and thereby
	\begin{equation}
		\label{eq:M_plus_psd}
		R_w^+ + R_s^+
		=
		-\Delta_1^T(\Lambda^{+} + \Lambda_s^{+})\Delta_1
		-
		(\Delta_2^+)^T\Gamma^{+}\Delta_2^+
		\le 0.
	\end{equation}
	The component forms of $R_s^+$ and $D_{ws}^+$ are given by
	\[
	R_s^+ u|_i
	= \lambda_{s,i}^{+}(u_{i+1}-u_i)
	-
	\lambda_{s,i-1}^{+}(u_i-u_{i-1})
	\]
	and
	\begin{equation}
		\label{eq:Dws_plus_pointwise}
		\begin{split}
			D_{ws}^+u|_i
			& = D_{w}^+u|_i + \frac1h R_s^+ u|_i \\
			& =
			\frac{\widehat u_i^{\,w,+}-\widehat u_{i-1}^{\,w,+}}{h}
			+
			\frac1h
			\left[
			\lambda_{s,i}^{+}(u_{i+1}-u_i)
			-
			\lambda_{s,i-1}^{+}(u_i-u_{i-1})
			\right].
		\end{split}
	\end{equation}
	Substituting \eqref{eq:Dws_minus_pointwise} and \eqref{eq:Dws_plus_pointwise} into \eqref{eq:Ddiss_1d} gives
	\begin{equation}\label{eq:Ddiss_pointwise}
		\begin{split}
			D_{\mathrm{diss}}u|_i
			& = \frac12 D_{ws}^-u|_i - \frac12 D_{ws}^+u|_i\\
			&=
			\frac{\widehat u_i^{\,w,-}-\widehat u_{i-1}^{\,w,-}}{2h}
			-
			\frac{\widehat u_i^{\,w,+}-\widehat u_{i-1}^{\,w,+}}{2h}
			\\
			&\quad
			+
			\frac{1}{2h}
			\left[
			(\lambda_{s,i}^{-}+\lambda_{s,i}^{+})(u_i-u_{i+1})
			-
			(\lambda_{s,i-1}^{-}+\lambda_{s,i-1}^{+})(u_{i-1}-u_i)
			\right].
		\end{split}
	\end{equation}
	
	\paragraph{Dissipativity and Accuracy.}
	We now show the dissipation property of  the 1D operator \eqref{eq:Ddiss_1d} and the corresponding 2D operator \(\mathbf{D}_{\mathrm{diss}}\), which is needed for the energy stability in Theorem~\ref{thm:euler_energy_stability}.
	
	\begin{thm}
		\label{thm:Ddiss_psd}
		Under periodic boundary conditions, the 1D dissipative operator \eqref{eq:Ddiss_1d} with $H=h I_N$ satisfies
		\begin{equation}\label{eq:1D_HD_dissi}
			\mathbf v^T HD_{\mathrm{diss}} \mathbf v \ge 0
			\qquad
			\forall \mathbf v \in \mathbb R^N.
		\end{equation}
		Therefore, the corresponding 2D dissipation operator \(\mathbf{D}_{\mathrm{diss}}\) satisfies the condition \eqref{eq:Ddiss_cond}, \ie, $\vec{U}^T{\mathbf{H}}\mathbf{D}_{\mathrm{diss}}\vec{U}\ge 0, \forall \vec{U} \in \mathbb R^{5N_{x}N_{y}}$.
		Here \(\mathbf{D}_{\mathrm{diss}}\) is defined as \eqref{eq:2d_diss},  $\mathbf{H} = I_5\otimes H_{x}\otimes H_{y}$, $H_{x} = \delta_x I_{N_x}$, $H_{y} = \delta_y I_{N_y}$, and $D_{\mathrm{diss},x}$ and $D_{\mathrm{diss},y}$ are the 1D dissipation operator $D_{\mathrm{diss}}$.
	\end{thm}
	
	\begin{proof}
		According to the decompositions of $D_{ws}^-$ and $D_{ws}^+$ in \eqref{eq:Dws_minus_def} and \eqref{eq:Dws_plus_def}, respectively, we have
		\begin{equation*}
			\begin{split}
				\mathbf v^T HD_{\mathrm{diss}} \mathbf v
				& =\mathbf v^T (Q_w^- + R_w^- + R_s^- - Q_w^+ - R_w^+ - R_s^+) \mathbf v \\
				& = \mathbf v^T ( R_w^- + R_s^- - R_w^+ - R_s^+) \mathbf v \\
				& = \mathbf v^T ( R_w^- + R_s^-)\mathbf v   - \mathbf v^T  (R_w^+ + R_s^+) \mathbf v \ge 0,
			\end{split}
		\end{equation*}
		where the second equality follows from the fact that $Q_w^-$ and $Q_w^+$ are skew-symmetric matrices. Thus \eqref{eq:1D_HD_dissi} holds.
		Furthermore, by setting $D_{\mathrm{diss},x} = D_{\mathrm{diss}} $ and using the property of Kronecker products, we have
		\[
		\begin{split}
			{\mathbf H}\mathbf D_{\mathrm{diss},x}
			&=
			(I_5\otimes H_x\otimes H_y)
			(I_5\otimes D_{\mathrm{diss},x}\otimes I_{N_y})  \\
			&= I_5\otimes (H_xD_{\mathrm{diss},x})\otimes H_y\\
			&= \delta_y I_5\otimes (H_xD_{\mathrm{diss},x})\otimes I_{N_y}.
		\end{split}
		\]
		From this and \eqref{eq:1D_HD_dissi}, it is easy to see that $\vec{U}^T{\mathbf{H}}\mathbf{D}_{\mathrm{diss},x}\vec{U}\ge 0$ for all $\vec{U} \in \mathbb R^{5N_{x}N_{y}}$. A similar result holds for $\mathbf{D}_{\mathrm{diss},y}$, and thus condition \eqref{eq:Ddiss_cond} is satisfied.
		
	\end{proof}

	Next we analyse the accuracy of the dissipation term for smooth solutions.
	For	the underlying linear third-order upwind pair $D^-$ and $D^+$ in \eqref{eq:D+D-}, Taylor expansion at \(x_i\) gives
	\[
	D^-u|_i
	=
	u_x+\frac1{12}h^3u_{xxxx}+O(h^4), \quad
	D^+u|_i
	=
	u_x-\frac1{12}h^3u_{xxxx}+O(h^4).
	\]
	Thus
	\begin{equation}
		\label{eq:Ddiss_linear_accuracy_p3}
		\frac12(D^- -D^+)u|_i
		=
		\frac1{12}h^3u_{xxxx}+O(h^4).
	\end{equation}
	We further show that the nonlinear WENO reconstruction does not change the order of the above error.
	Indeed, for smooth solutions, we have
	\[
	\beta_i^{(1),-}
	=
	u_x^2h^2+u_xu_{xx}h^3
	+
	\left(\frac14u_{xx}^2+\frac13u_xu_{xxx}\right)h^4
	+
	O(h^5),
	\]
	\[
	\beta_i^{(2),-}
	=
	u_x^2h^2-u_xu_{xx}h^3
	+
	\left(\frac14u_{xx}^2+\frac13u_xu_{xxx}\right)h^4
	+
	O(h^5),
	\]
	and
	\[
	\tau_i^-
	=
	\left(u_{i+1}-2u_i+u_{i-1}\right)^2
	=
	u_{xx}^2h^4+O(h^6).
	\]
	For sufficiently small $\varepsilon$ (\(\varepsilon=O(h^2)\)), the ratios
	\(\tau_i^-/(\varepsilon+\beta_i^{(1),-})\) and
	\(\tau_i^-/(\varepsilon+\beta_i^{(2),-})\) share the same leading
	\(O(h^2)\) term and differ only by \(O(h^3)\). Hence
	\begin{equation}\label{eq:w-d}
		w_i^{(r),-}
		-
		d^{(r)} = O(h^3),
		\qquad r=1,2.
	\end{equation}
	On the other hand, noting that
	$
	\sum_{r=1}^{2}\left(w_i^{(r),-}-d^{(r)}\right)=0,
	$
	we obtain
	\begin{equation*}
		\begin{split}
			\widehat u_i^{\,w,-}-\widehat u_i
			&= \left(w_i^{(1),-}-d^{(1)}\right) \widehat u_i^{(1),-} + \left(w_i^{(2),-}-d^{(2)}\right) \widehat u_i^{(2),-}\\
			&=
			\left(w_i^{(1),-}-d^{(1)}\right)
			\left(\widehat u_i^{(1),-}-\widehat u_i^{(2),-}\right)\\
			&= \frac12 \left(w_i^{(1),-}-d^{(1)}\right)
			\left(u_{i-1}-2u_i+u_{i+1}\right) = O(h^5).
		\end{split}
	\end{equation*}
	It then follows that
	\[
	D_w^-u|_i-D^-u|_i
	=
	\frac{
		\left(\widehat u_i^{\,w,-}-\widehat u_i\right)
		-
		\left(\widehat u_{i-1}^{\,w,-}-\widehat u_{i-1}\right)}
	{h}
	=
	O(h^4).
	\]
	Similarly, one can show that $D_w^+u|_i-D^+u|_i=O(h^4)$ for smooth solutions.
	Combining these and \eqref{eq:Ddiss_linear_accuracy_p3}, we have
	\begin{equation}
		\label{eq:Draw_accuracy_p3}
		\frac12(D_w^- - D_w^+)u|_i
		=
		\frac1{12}h^3u_{xxxx}+O(h^4).
	\end{equation}

	It remains to account for the stabilization term in \eqref{eq:Ddiss_pointwise}, \ie,
	\begin{equation}\label{eq:stabilization}
		\frac{1}{2h}
		\left[
		(\lambda_{s,i}^{-}+\lambda_{s,i}^{+})(u_i-u_{i+1})
		-
		(\lambda_{s,i-1}^{-}+\lambda_{s,i-1}^{+})(u_{i-1}-u_i)
		\right].
	\end{equation}
	It follows from \eqref{eq:w-d} that	\(w_i^{(2),-}-w_{i+1}^{(2),-}=O(h^3)\), and thereby
	$
	\lambda_i^-
	=
	\frac14\left(w_i^{(2),-}-w_{i+1}^{(2),-}\right)
	=
	O(h^3).
	$
	This further yields \(\lambda_{s,i}^-=O(h^3)\) since $\delta_1 = h^3$. Similarly, it can be shown that \(\lambda_{s,i}^+=O(h^3)\). With these and \(u_i-u_{i+1}=O(h)\), we see that the stabilization term in \eqref{eq:stabilization} is also \(O(h^3)\) in smooth regions.

	From this and \eqref{eq:Ddiss_linear_accuracy_p3}, we conclude that for smooth solutions the dissipation term $D_{\mathrm{diss}}u|_i$ in \eqref{eq:Dws_minus_pointwise} is $O(h^3)$, which is one order lower than the fourth-order scheme \eqref{eq:scheme_without_stabilization} without the dissipation term. The same conclusion also holds for the 2D scheme \eqref{eq:euler_pointwise}, \ie, its truncation error is $O(h^3)$, where the error of the scheme without dissipation is $O(h^4)$ and that of the dissipation term is $O(h^3)$.

	\begin{rem}
		Unlike the work \cite{JiangWang2024UpwindSBP}, where the upwind operators \(D_{ws}^-\) and \(D_{ws}^+\) are applied to the fluxes of the hyperbolic equations, here we use them to construct the dissipation operator $D_{\mathrm{diss}}$.
		As analyzed above, for smooth solutions the dissipation term with $D_{\mathrm{diss}}$ is $O(h^3)$, leading to the third-order accuracy of the final scheme. For discontinuous problems, thanks to the WENO reconstruction, it can efficiently suppress the oscillations of the original scheme without the dissipation term. These properties are demonstrated in our numerical examples.
		Additionally, the final semi-discrete scheme is provably energy stable with this dissipation term.
	\end{rem}

	\section{Energy stable schemes for reactive Navier--Stokes equations}
	\label{sec:ns_discrete}
	\setcounter{equation}{0}
	
	
	In this section we construct an energy stable scheme for the reactive Navier--Stokes equations based on the skew-symmetric form \eqref{eq:final_form_ns}. For the convective part $(A_1\Phi)_{x_1} + A_2\Phi_{x_1} + (B_1\Phi)_{x_2} + B_2\Phi_{x_2}$, the discretizations are the same as those for the reactive Euler equations. Thus we only need to further discretize the diffusion term $S$, whose specific expressions are given in \eqref{eq:rns_primitive}.
	
	\subsection{Spatial discretization}
	
	Using the same notation as in \eqref{eq:euler_semidiscrete} for the reactive Euler equations, the semi-discrete finite difference scheme for the reactive Navier--Stokes equations \eqref{eq:final_form_ns} under periodic boundary conditions can be written as
	\begin{equation}
		\label{eq:ns_semidiscrete}
		\vec{\Phi}_t
		+
		\mathbf{D}_{x}(\mathbf{A}_1\vec{\Phi})
		+
		\mathbf{A}_2\mathbf{D}_{x}\vec{\Phi}
		+
		\mathbf{D}_{y}(\mathbf{B}_1\vec{\Phi})
		+
		\mathbf{B}_2\mathbf{D}_{y}\vec{\Phi}
		+
		\mathbf{C}\vec{\Phi}
		=
		-\mathbf{D}_{\mathrm{diss}}\vec{\Phi}
		+
		\mathbf{\Lambda}\vec{S}.
	\end{equation}
	Here \(\mathbf{\Lambda} \in \mathbb R^{5N_{x}N_{y} \times 5N_{x}N_{y}}\) and $\vec{S} \in \mathbb R^{5N_{x}N_{y}}$ are the discrete counterparts
	of the continuous \(\Lambda\) and $S$ in
	\eqref{eq:final_form_ns}, respectively. Specifically, \(\mathbf{\Lambda}\) has the block-diagonal structure
	\begin{equation}\label{eq:Lambda_discrete}
		\mathbf{\Lambda}
		=
		\begin{pmatrix}
			\boldsymbol{\lambda}_1 & & & & \\
			& \boldsymbol{\lambda}_2 & & & \\
			& & \boldsymbol{\lambda}_3 & & \\
			& & & \boldsymbol{\lambda}_4 & \\
			& & & & \boldsymbol{\lambda}_5
		\end{pmatrix},
		\qquad
		\boldsymbol{\lambda}_k
		=
		\operatorname{diag}\bigl(
		\lambda_k(x_{1},y_{1}),\dots,
		\lambda_k(x_{N_x},y_{N_y})
		\bigr),
	\end{equation}
	where \(\lambda_k\) denotes the \(k\)-th diagonal entry of \(\Lambda\), \ie,
	\(\lambda_1 = 1/(2\phi_1)\), \(\lambda_2 = \phi_1\), \(\lambda_3 = \phi_1\),
	\(\lambda_4 = 1/(2\phi_4)\), \(\lambda_5 = 1/(2\phi_5)\).
	
	
	Using the same SBP operators $D_x$ and $D_y$ as in Section~\ref{sec:subsection_operator}, both taken as the operator $D$ in \eqref{eq:Dskew_1d} to approximate $\partial_x$ and $\partial_y$, respectively,
	we denote
	\begin{equation}\label{eq:mathcal_DxDy}
		\mathcal{D}_x = D_x\otimes I_{N_y} \in \mathbb R^{N_{x}N_{y}},
		\qquad
		\mathcal{D}_y = I_{N_x}\otimes D_y \in \mathbb R^{N_{x}N_{y}}.
	\end{equation}
	Then the diffusion term \(\mathbf{\Lambda}\vec S \in \mathbb R^{5N_{x}N_{y}} \) is given by
	\begin{equation}\label{eq:LambdaS}
		\mathbf{\Lambda}\vec{S}
		=
		\begin{pmatrix}
			0\\[0.4em]
			\operatorname{diag}(\vec{\Phi}_1)^{-1}
			\bigl(
			\mathcal{D}_{x}\vec{\tau}_{11}
			+
			\mathcal{D}_{y}\vec{\tau}_{12}
			\bigr)\\[0.7em]
			\operatorname{diag}(\vec{\Phi}_1)^{-1}
			\bigl(
			\mathcal{D}_{x}\vec{\tau}_{21}
			+
			\mathcal{D}_{y}\vec{\tau}_{22}
			\bigr)\\[0.7em]
			\dfrac{\gamma-1}{2}
			\operatorname{diag}(\vec{\Phi}_4)^{-1}
			\Bigl[
			\mathcal{D}_{x}\bigl(\kappa\mathcal{D}_{x}\vec{T}+q\sigma\mathcal{D}_{x}\vec{Y}\bigr)
			+
			\mathcal{D}_{y}\bigl(\kappa\mathcal{D}_{y}\vec{T}+q\sigma\mathcal{D}_{y}\vec{Y}\bigr)
			+
			\vec{\Psi}
			\Bigr]\\[0.9em]
			\dfrac12
			\operatorname{diag}(\vec{\Phi}_5)^{-1}
			\bigl(
			\mathcal{D}_{x}\bigl(\sigma\mathcal{D}_{x}\vec{Y}\bigr)
			+
			\mathcal{D}_{y}\bigl(\sigma\mathcal{D}_{y}\vec{Y}\bigr)
			\bigr)
		\end{pmatrix}.
	\end{equation}
	Note that the primitive variables in the above expressions are computed from the transformed state vector \(\vec{\Phi}\). More precisely, for the \( \ell \)-th component of the grid vector, we compute
	\[
	u_{1,\ell}=\frac{\Phi_{2,\ell}}{\Phi_{1,\ell}},\qquad
	u_{2,\ell}=\frac{\Phi_{3,\ell}}{\Phi_{1,\ell}},\qquad
	Y_\ell=\frac{\Phi_{5,\ell}^{2}}{\Phi_{1,\ell}^{2}}, \qquad
	T_\ell
	=
	\frac{
		\Phi_{4,\ell}^{2}
		-
		q(\gamma-1)\Phi_{5,\ell}^{2}
	}{
		\Phi_{1,\ell}^{2}
	}.
	\]

	\subsection{Energy stability}
	
	Next we show the energy stability of the semi-discrete scheme \eqref{eq:ns_semidiscrete} for the reactive Navier--Stokes equations. With the notation $B_x, B_y, H_x, H_y$ in \eqref{eq:HB}, which corresponds to the matrices in the SBP operators $D_x$ and $D_y$, we introduce
	\[
	\mathcal{H} = H_x\otimes H_y,\qquad
	\mathcal{B}_x = B_x\otimes H_y,\qquad
	\mathcal{B}_y = H_x\otimes B_y.
	\]
	Then \({\mathbf{H}} = I_5\otimes\mathcal{H}\),
	\(\mathbf{B}_x = I_5\otimes\mathcal{B}_x\), and
	\(\mathbf{B}_y = I_5\otimes\mathcal{B}_y\).
	Our stability result for the scheme \eqref{eq:ns_semidiscrete} reads as follows.
	
	\begin{thm}
		\label{thm:ns_energy_stability}
		For the semi-discrete scheme \eqref{eq:ns_semidiscrete} with periodic boundary conditions, suppose that the conditions in Theorem~\ref{thm:euler_energy_stability} hold, \ie, the first-order derivative operators \(\mathbf{D}_{x}\) and \(\mathbf{D}_{y}\)
		satisfy the SBP identities \eqref{eq:sbp_vector_x1} and \eqref{eq:sbp_vector_x2}, respectively,
		and the operator \(\mathbf{D}_{\mathrm{diss}}\) admits the dissipation property \eqref{eq:Ddiss_cond}.
		Furthermore, for the discretization of the diffusion term in \eqref{eq:LambdaS}, we assume that the corresponding 1D derivative operators $D_x, D_y$ in \eqref{eq:mathcal_DxDy} are both SBP operators.
		Then the scheme \eqref{eq:ns_semidiscrete} is energy stable in the sense that
		\begin{equation}
			\label{eq:ns_energy_diss_general}
			\frac{\mathrm{d}}{\mathrm{d}t}
			\|\vec{\Phi}\|_{\mathbf{P}{\mathbf{H}}}^2
			\le 0.
		\end{equation}
	\end{thm}
	
	\begin{proof}
		Multiplying \eqref{eq:ns_semidiscrete} by \(2\vec{\Phi}^{T}\mathbf P{\mathbf H}\) and following the proof of Theorem~\ref{thm:euler_energy_stability}, we obtain
		\[
		\frac{\mathrm{d}}{\mathrm{d}t}\|\vec{\Phi}\|_{\mathbf{P}{\mathbf{H}}}^2
		=
		-2\vec{\Phi}^{T}\mathbf{P}{\mathbf{H}}
		\bigl(\mathbf{C}+\mathbf{D}_{\mathrm{diss}}\bigr)\vec{\Phi}
		+
		2\vec{\Phi}^{T}\mathbf{P}{\mathbf{H}}\mathbf{\Lambda}\vec{S}.
		\]
		According to Theorem~\ref{thm:euler_energy_stability}, we have $-2\vec{\Phi}^{T}\mathbf{P}{\mathbf{H}}
		\bigl(\mathbf{C}+\mathbf{D}_{\mathrm{diss}}\bigr)\vec{\Phi} \le 0$. Then we turn to show that $\vec{\Phi}^{T}\mathbf{P}{\mathbf{H}}\mathbf{\Lambda}\vec{S} \le 0$.
		
		Recall the definitions of \(\mathbf{P}\) and \(\mathbf{H}\) in \eqref{eq:P_mathbf} and \eqref{eq:HB}, respectively.
		Similar to the continuous identity \eqref{eq:Phi_PLS}, we compute
		\begin{equation}
			\label{eq:ns_viscous_boundary}
			\begin{aligned}
				2\vec{\Phi}^{T}(\mathbf{P}{\mathbf{H}}\mathbf{\Lambda}\vec{S})
				=
				&(\gamma-1)\Big[
				\vec{u}_1^{T}\mathcal{H}\mathcal{D}_{x}\vec{\tau}_{11}
				+
				\vec{\tau}_{11}^{T}\mathcal{H}\mathcal{D}_{x}\vec{u}_1
				+
				\vec{u}_2^{T}\mathcal{H}\mathcal{D}_{x}\vec{\tau}_{21}
				+
				\vec{\tau}_{21}^{T}\mathcal{H}\mathcal{D}_{x}\vec{u}_2
				\\
				&+
				\mathbf{1}^{T}\mathcal{H}\mathcal{D}_{x}
				\bigl(\kappa\mathcal{D}_{x}\vec{T}\bigr)
				+
				\Bigl(q+\frac{1}{\gamma-1}\Bigr)
				\mathbf{1}^{T}\mathcal{H}\mathcal{D}_{x}
				\bigl(\sigma\mathcal{D}_{x}\vec{Y}\bigr)
				\\
				&+
				\vec{u}_1^{T}\mathcal{H}\mathcal{D}_{y}\vec{\tau}_{12}
				+
				\vec{\tau}_{12}^{T}\mathcal{H}\mathcal{D}_{y}\vec{u}_1
				+
				\vec{u}_2^{T}\mathcal{H}\mathcal{D}_{y}\vec{\tau}_{22}
				+
				\vec{\tau}_{22}^{T}\mathcal{H}\mathcal{D}_{y}\vec{u}_2
				\\
				&+
				\mathbf{1}^{T}\mathcal{H}\mathcal{D}_{y}
				\bigl(\kappa\mathcal{D}_{y}\vec{T}\bigr)
				+
				\Bigl(q+\frac{1}{\gamma-1}\Bigr)
				\mathbf{1}^{T}\mathcal{H}\mathcal{D}_{y}
				\bigl(\sigma\mathcal{D}_{y}\vec{Y}\bigr)
				\Big],
			\end{aligned}
		\end{equation}
		where $\mathbf{1} = (1,1,...,1)^T \in \mathbb{R}^{N_x N_y}$.
		
		Using the 1D SBP identities $H_x D_x + D_x^T H_x = B_x$ and $H_y D_y + D_y^T H_y = B_y$, we have
		\[
		\mathcal H\mathcal D_x+\mathcal D_x^T\mathcal H=\mathcal B_x,
		\qquad
		\mathcal H\mathcal D_y+\mathcal D_y^T\mathcal H=\mathcal B_y.
		\]
		Then it follows that
		\[
		\begin{aligned}
			\vec{u}_1^{T}\mathcal{H}\mathcal{D}_{x}\vec{\tau}_{11}
			+
			\vec{\tau}_{11}^{T}\mathcal{H}\mathcal{D}_{x}\vec{u}_1
			& =
			\vec{u}_1^{T}
			\bigl(
			\mathcal{H}\mathcal{D}_{x}
			+
			\mathcal{D}_{x}^{T}\mathcal{H}
			\bigr)
			\vec{\tau}_{11}
			=
			\vec{u}_1^{T}\mathcal{B}_{x}\vec{\tau}_{11}, \\
			\vec{u}_2^{T}\mathcal{H}\mathcal{D}_{x}\vec{\tau}_{21}
			+
			\vec{\tau}_{21}^{T}\mathcal{H}\mathcal{D}_{x}\vec{u}_2
			& =
			\vec{u}_2^{T}
			\bigl(
			\mathcal{H}\mathcal{D}_{x}
			+
			\mathcal{D}_{x}^{T}\mathcal{H}
			\bigr)
			\vec{\tau}_{21}
			=
			\vec{u}_2^{T}\mathcal{B}_{x}\vec{\tau}_{21}, \\
			\vec{u}_1^{T}\mathcal{H}\mathcal{D}_{y}\vec{\tau}_{12}
			+
			\vec{\tau}_{12}^{T}\mathcal{H}\mathcal{D}_{y}\vec{u}_1
			& =
			\vec{u}_1^{T}
			\bigl(
			\mathcal{H}\mathcal{D}_{y}
			+
			\mathcal{D}_{y}^{T}\mathcal{H}
			\bigr)
			\vec{\tau}_{12}
			=
			\vec{u}_1^{T}\mathcal{B}_{y}\vec{\tau}_{12}, \\
			\vec{u}_2^{T}\mathcal{H}\mathcal{D}_{y}\vec{\tau}_{22}
			+
			\vec{\tau}_{22}^{T}\mathcal{H}\mathcal{D}_{y}\vec{u}_2
			& =
			\vec{u}_2^{T}
			\bigl(
			\mathcal{H}\mathcal{D}_{y}
			+
			\mathcal{D}_{y}^{T}\mathcal{H}
			\bigr)
			\vec{\tau}_{22}
			=
			\vec{u}_2^{T}\mathcal{B}_{y}\vec{\tau}_{22}.
		\end{aligned}
		\]
		On the other hand, since $D_x$ and $D_y$ are consistent derivative operators, we have \(\mathcal D_x\mathbf 1=\mathcal D_y\mathbf 1=0\) (this can also be directly verified by the definition of $D=D^- + D^+$ with $D^-,D^+$ in \eqref{eq:D+D-}), which further yields
		\[
		\begin{aligned}
			\mathbf 1^T\mathcal H\mathcal D_x
			\bigl(\kappa\mathcal D_x\vec T\bigr)
			=
			\mathbf 1^T
			\bigl(
			\mathcal H\mathcal D_x
			+
			\mathcal D_x^T\mathcal H
			\bigr)
			\bigl(\kappa\mathcal D_x\vec T\bigr)
			=
			\mathbf 1^T\mathcal B_x
			\bigl(\kappa\mathcal D_x\vec T\bigr).
		\end{aligned}
		\]
		Similarly,
		\begin{equation*}
			\begin{split}
				\mathbf{1}^{T}\mathcal{H}\mathcal{D}_{x}
				\bigl(\sigma\mathcal{D}_{x}\vec{Y}\bigr)
				& = \mathbf{1}^{T}\mathcal{B}_{x}
				\bigl(\sigma\mathcal{D}_{x}\vec{Y}\bigr) ,\\
				\mathbf{1}^{T}\mathcal{H}\mathcal{D}_{y}
				\bigl(\kappa\mathcal{D}_{y}\vec{T}\bigr)
				& = \mathbf{1}^{T}\mathcal{B}_{y}
				\bigl(\kappa\mathcal{D}_{y}\vec{T}\bigr), \\
				\mathbf{1}^{T}\mathcal{H}\mathcal{D}_{y}
				\bigl(\sigma\mathcal{D}_{y}\vec{Y}\bigr)
				& = \mathbf{1}^{T}\mathcal{B}_{y}
				\bigl(\sigma\mathcal{D}_{y}\vec{Y}\bigr).
			\end{split}
		\end{equation*}
		Substituting the above identities into \eqref{eq:ns_viscous_boundary} gives
		\begin{equation}
			\begin{aligned}
				2\vec{\Phi}^{T}(\mathbf{P}{\mathbf{H}}\mathbf{\Lambda}\vec{S})
				=&
				(\gamma-1)\Big[
				\vec{u}_1^{T}\mathcal{B}_{x}\vec{\tau}_{11}
				+
				\vec{u}_2^{T}\mathcal{B}_{x}\vec{\tau}_{21}
				+
				\mathbf{1}^{T}\mathcal{B}_{x}
				\bigl(\kappa\mathcal{D}_{x}\vec{T}\bigr)\\
				&+
				\Bigl(q+\frac{1}{\gamma-1}\Bigr)
				\mathbf{1}^{T}\mathcal{B}_{x}
				\bigl(\sigma\mathcal{D}_{x}\vec{Y}\bigr)\\
				&+
				\vec{u}_1^{T}\mathcal{B}_{y}\vec{\tau}_{12}
				+
				\vec{u}_2^{T}\mathcal{B}_{y}\vec{\tau}_{22}
				+
				\mathbf{1}^{T}\mathcal{B}_{y}
				\bigl(\kappa\mathcal{D}_{y}\vec{T}\bigr)\\
				&+
				\Bigl(q+\frac{1}{\gamma-1}\Bigr)
				\mathbf{1}^{T}\mathcal{B}_{y}
				\bigl(\sigma\mathcal{D}_{y}\vec{Y}\bigr)
				\Big],
			\end{aligned}
		\end{equation}
		which is an analogue of \eqref{eq:Phi_PLS}.
		Note that under periodic boundary conditions, \(\mathcal B_x=0 \) and \(\mathcal B_y=0\). Thus
		$
		2\vec{\Phi}^{T}\mathbf{P}{\mathbf{H}}\mathbf{\Lambda}\vec{S}=0
		$
		and \eqref{eq:ns_energy_diss_general} holds.
	\end{proof}

	\section{Numerical experiments}
	\label{sec:ne}
	\setcounter{equation}{0}
	In this section, we demonstrate the proposed schemes by several one- and two-dimensional examples. These include smooth periodic problems, a classical shock tube, double rarefactions, shock--contact interactions, periodic vortex transport, and 2D detonation-type configurations; related benchmark setups can be found in \cite{ZhangShu2010PositivityDG,Sod1978Survey,ShuOsher1989ENO,LindeRoe1997RobustEuler,SchulzRinneCollinsGlaz1993Riemann2D,LiskaWendroff2003Comparison,YeeSandhamDjomehri1999Filters,WangZhangShuNing2012DetonationDG,Zhang2017PositivityDGNS}. The spatial discretization is given by our energy stable schemes, and the time integration is performed by the third-order strong stability preserving Runge--Kutta method \cite{GottliebShu1998SSP}, since our schemes are third-order in space.
	The chemical reaction term is taken as the usual Arrhenius form
	\[
	\omega
	=
	-K\rho Y\exp\left(-\frac{T_a}{T}\right),
	\]
	where \(Y\in[0,1]\) is the reactant mass fraction, \(K>0\) is the rate coefficient, \(T_a>0\) is the activation temperature, \(T=\frac{p}{\rho}\) is the nondimensional temperature, and \(q\) denotes the heat release in the total energy.
	
	For the computation of the reactive Navier-Stokes equations, we set \(\mu=1/\mathrm{Re}\) with $\mathrm{Re}$ the Reynolds number and use Stokes' hypothesis \(\lambda=-2\mu/3\).
	The time step for 1D problems is chosen according to the convective CFL condition
	\[
	\delta_t
	=
	\mathrm{CFL}\frac{\delta_x}{\alpha},
	\qquad
	\alpha=\max_i(|u_i|+c_i),
	\qquad
	c_i=\sqrt{\frac{\gamma p_i}{\rho_i}}.
	\]
	%
	In 2D case, the CFL number is defined by
	\[
	\delta_t
	=
	\frac{\mathrm{CFL}}{\alpha_x/\delta_x+\alpha_y/\delta_y},
	\qquad
	\alpha_x=\max_{i,j}(|u_{i,j}|+c_{i,j}),
	\qquad
	\alpha_y=\max_{i,j}(|v_{i,j}|+c_{i,j}),
	\]
	where \(u=u_1\), \(v=u_2\), and
	$
	c_{i,j}=\sqrt{\frac{\gamma p_{i,j}}{\rho_{i,j}}}.
	$

	\subsection{1D examples}
	
	\begin{example}\label{ex:1d_smooth}
		We first verify the accuracy of the proposed method by a 1D smooth problem adapted from \cite{ZhangShu2010PositivityDG}. The initial data are
		\[
		\rho(x,0) = 1 - 0.99 \sin(x), \quad
		u(x,0) = 1, \quad
		p(x,0) = 0.01, \quad
		Y(x,0) = 0.5+0.49\sin(x).
		\]
		The parameters are \(\gamma=1.4\), \(q=50\), \(K=145.68913\), and \(T_a=50\). The computational domain is $[0,2\pi]$, and periodic boundary conditions are imposed. Since no analytical solution is available for this problem, the numerical solution computed on the refined mesh with $\delta_x=\pi/2048$ is used as the reference solution. The errors at $t=0.02$ are listed in Table~\ref{tab:error1d}, where the expected convergence behavior is observed.
		Figure~\ref{fig:energy_periodic_1d} shows that the discrete energy decreases monotonically, in agreement with the semi-discrete energy estimate proved in Section~\ref{sec:es}.
	\end{example}
	
	\begin{table}[htbp]
		\centering
		\begin{threeparttable}
			\caption{\textbf{Example~\ref{ex:1d_smooth}:} Error table for the 1D smooth problem.}
			\label{tab:error1d}
			\small
			\setlength{\tabcolsep}{7pt}
			\begin{tabular}{c cc cc cc}
				\toprule
				$\delta_x$ & $L^1$ error & Order & $L^2$ error & Order & $L^\infty$ error & Order 	\\
				\midrule
				$\pi/64$  & 8.31e-06 & -    & 9.65e-06 & -    & 2.36e-05 & -   \\
				$\pi/128$ & 1.56e-06 & 2.41 & 2.10e-06 & 2.20 & 6.34e-06 & 1.90 \\
				$\pi/256$ & 1.63e-07 & 3.25 & 2.56e-07 & 3.04 & 1.07e-06 & 2.55 \\
				$\pi/512$ & 1.47e-08 & 3.47 & 2.12e-08 & 3.59 & 9.42e-08 & 3.52 \\
				\bottomrule
			\end{tabular}
		\end{threeparttable}
	\end{table}
	\begin{figure}[H]
		\centering
		\includegraphics[width=0.6\textwidth]{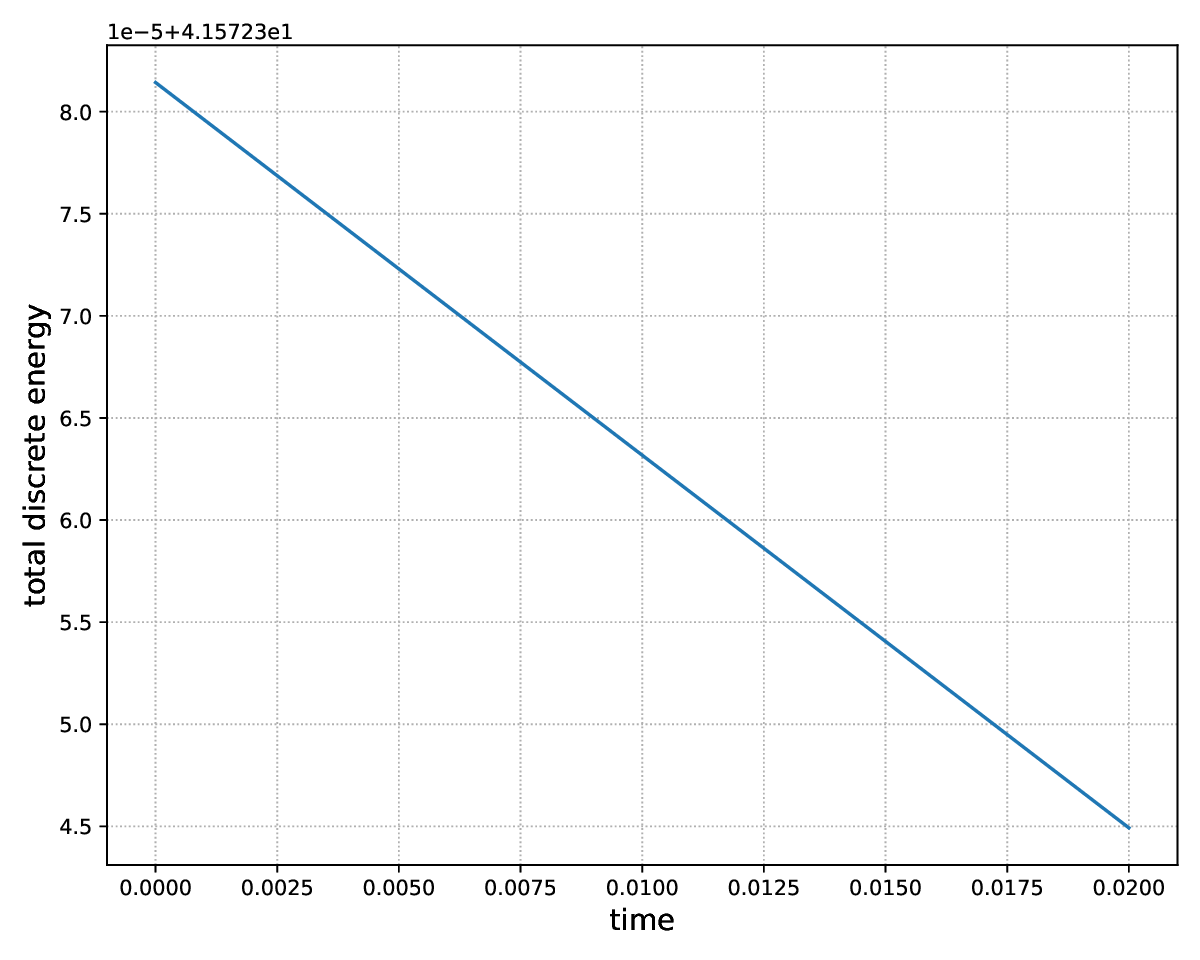}
		\caption{Example~\ref{ex:1d_smooth}: Discrete energy history for the 1D periodic problem.}
		\label{fig:energy_periodic_1d}
	\end{figure}
	
	\begin{example}\label{ex:sod}
		Next we solve the Sod shock tube problem for the Euler equations, \ie, $\omega=0$. The initial data are taken from the original test of Sod \cite{Sod1978Survey}:
		\[
		(\rho, u, p, Y) =
		\begin{cases}
			(1, 0, 1, 0), & x < 0, \\
			(0.125, 0, 0.1, 0), & x \ge 0.
		\end{cases}
		\]
		The computational domain is $[-1,1]$, the final time is $t=0.1$, and \(\gamma=1.4\). Since the waves do not reach the boundaries before the final time, the boundary values are kept as their initial states. We set \(\mathrm{CFL}=0.8\) and $\delta_x=1/200$. To illustrate the role of the artificial dissipation, we compute two solutions, one with ${D}_{\mathrm{diss}}$ and one without it. Figure~\ref{fig:sod_results} compares both results with the exact solution. Without artificial dissipation, visible oscillations appear near the discontinuities, especially around the contact discontinuity and the shock. With the dissipation operator, these oscillations are effectively suppressed, while the rarefaction wave, contact discontinuity, and shock are still sharply resolved.
	\end{example}

	\begin{figure}[htbp]
		\centering
		\begin{subfigure}[b]{0.49\linewidth}
			\centering
			\includegraphics[width=\linewidth]{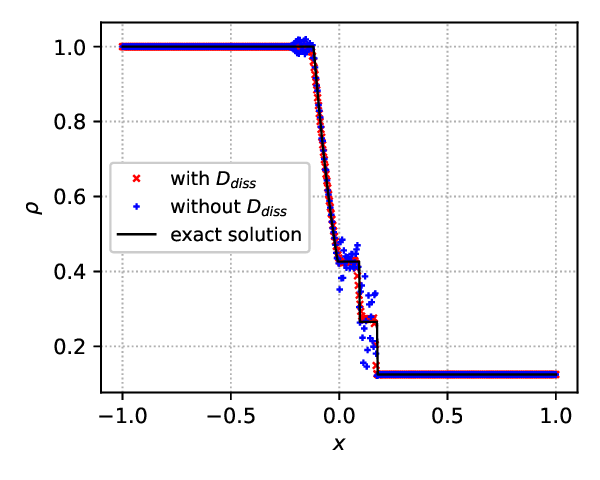}
			\caption{Density}
			\label{fig:sod_rho}
		\end{subfigure}
		\hfill
		\begin{subfigure}[b]{0.49\linewidth}
			\centering
			\includegraphics[width=\linewidth]{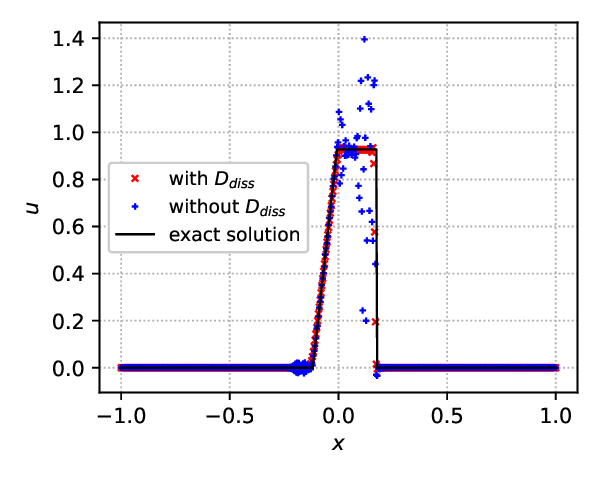}
			\caption{Velocity}
			\label{fig:sod_u}
		\end{subfigure}
		
		\vspace{0.35cm}
		
		\begin{subfigure}[b]{0.49\linewidth}
			\centering
			\includegraphics[width=\linewidth]{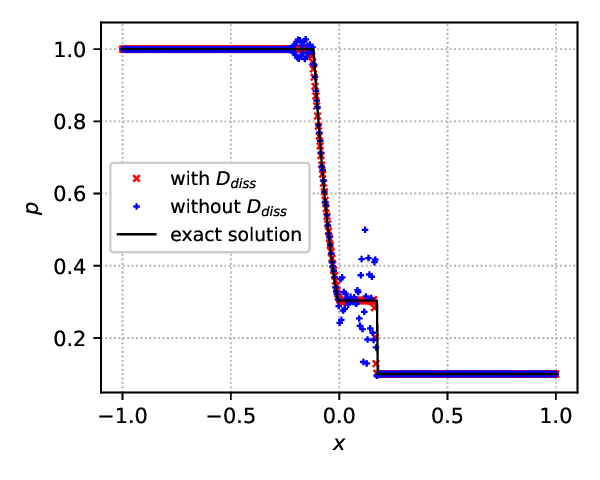}
			\caption{Pressure}
			\label{fig:sod_p}
		\end{subfigure}
		
		\caption{Example~\ref{ex:sod}: Numerical results of the Sod shock tube problem at $t=0.1$ with $\delta_x=\frac{1}{200}$.}
		\label{fig:sod_results}
	\end{figure}

	\begin{example}\label{ex:rarefaction}
		We also consider the double-rarefaction problem with a central vacuum/near-vacuum state \cite{LindeRoe1997RobustEuler} for the non-reactive Euler and Navier--Stokes equations ($\omega=0$). This test is used to examine robustness in the presence of very low density and pressure. The initial data are
		\[
		(\rho, u, p, Y) =
		\begin{cases}
			(7, -1, 0.2, 0), & x < 0, \\
			(7, 1, 0.2, 0), & x \ge 0,
		\end{cases}
		\]
		and we set \(\gamma=1.4\). The computational domain is $[-1,1]$, the final time is $t=0.6$, and \(\mathrm{CFL}=0.8\). Since the waves remain away from the boundaries, the boundary values are fixed at their initial states. For the viscous term, we take $\mathrm{Re}=1000$ as in \cite{Zhang2017PositivityDGNS}. Additionally, we set \(\Pr=0.72\), \(\mu=1/\mathrm{Re}\), \(\lambda=-2\mu/3\), \(\kappa=\mu c_p/\Pr\), \(c_p=\gamma/(\gamma-1)\), and \(\sigma=0\). Figure~\ref{fig:rarefaction_results} shows that the Euler computation agrees well with the exact Euler Riemann solution. The Navier--Stokes computation is compared with the same Euler exact solution only as an inviscid reference; the reference is not an exact solution of the viscous problem. No evident spurious oscillations are observed near the low-density region.
	\end{example}
	
	\begin{figure}[htbp]
		\centering
		\begin{subfigure}[b]{0.49\textwidth}
			\centering
			\includegraphics[width=\textwidth]{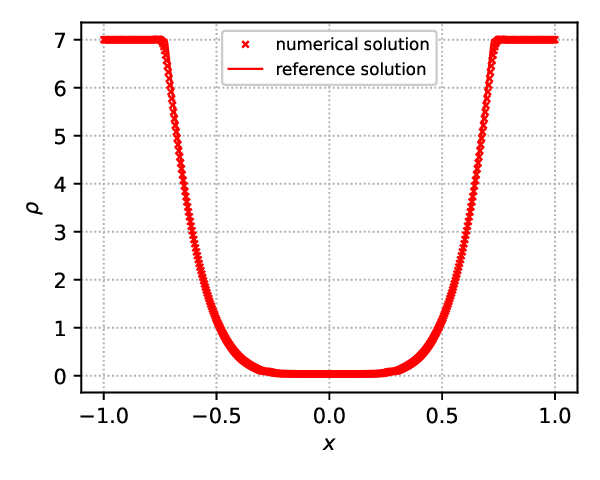}
			\caption{Density (Euler)}
		\end{subfigure}
		\hfill
		\begin{subfigure}[b]{0.49\textwidth}
			\centering
			\includegraphics[width=\textwidth]{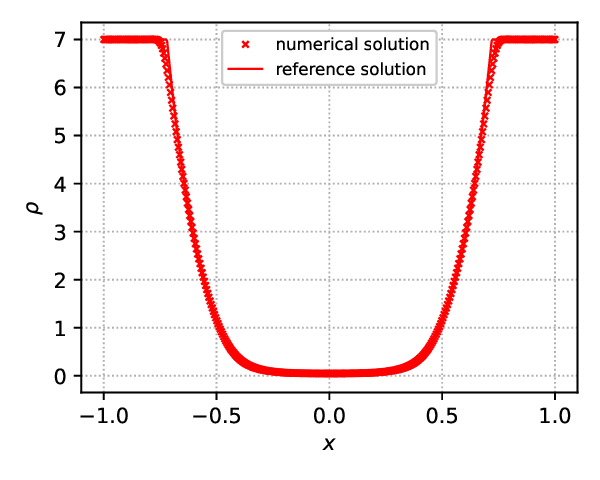}
			\caption{Density (Navier-Stokes)}
		\end{subfigure}
		
		\vspace{0.3cm}
		
		\begin{subfigure}[b]{0.49\textwidth}
			\centering
			\includegraphics[width=\textwidth]{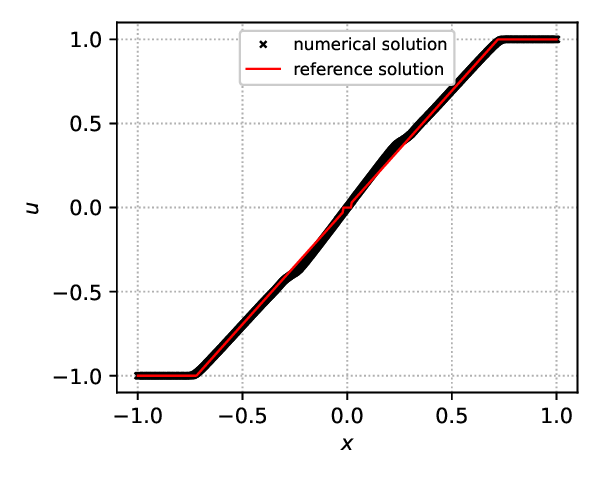}
			\caption{Velocity (Euler)}
		\end{subfigure}
		\hfill
		\begin{subfigure}[b]{0.49\textwidth}
			\centering
			\includegraphics[width=\textwidth]{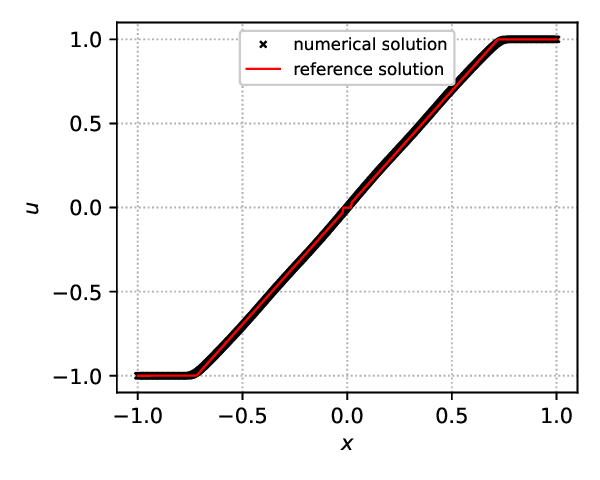}
			\caption{Velocity (Navier-Stokes)}
		\end{subfigure}
		
		\vspace{0.3cm}
		
		\begin{subfigure}[b]{0.49\textwidth}
			\centering
			\includegraphics[width=\textwidth]{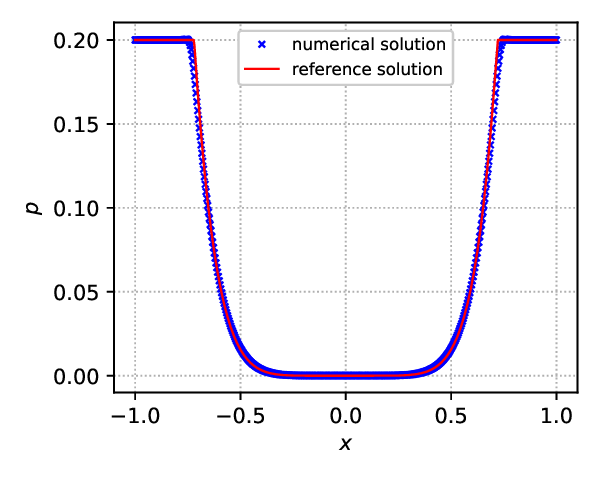}
			\caption{Pressure (Euler)}
		\end{subfigure}
		\hfill
		\begin{subfigure}[b]{0.49\textwidth}
			\centering
			\includegraphics[width=\textwidth]{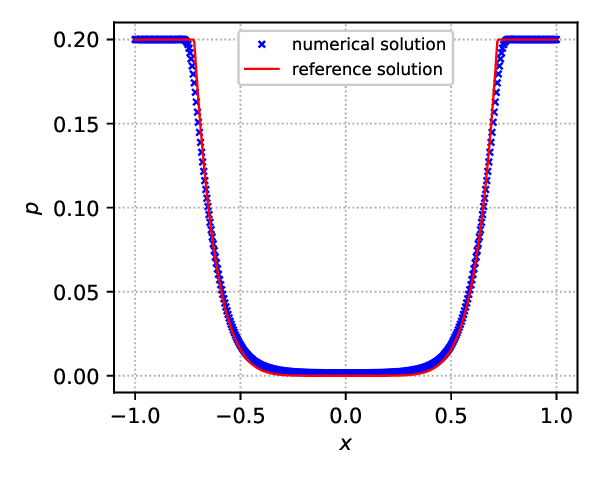}
			\caption{Pressure (Navier-Stokes)}
		\end{subfigure}
		
		\caption{Example~\ref{ex:rarefaction}: Numerical results of the double rarefaction problem at $t=0.6$ with $\delta_x=\frac{1}{200}$. Left: solutions of the Euler equations, Right: solutions of the Navier-Stokes equations with $\mathrm{Re}=1000$.}
		\label{fig:rarefaction_results}
	\end{figure}

	\subsection{2D examples}
	In the following we simulate 2D problems and write $u=u_1$ and $v=u_2$ for convenience.
	
	\begin{example}\label{ex:2d_smooth}
		We first verify the accuracy of the method by a 2D smooth problem adapted from \cite{ZhangShu2010PositivityDG}. The exact solution is
		\[
		\rho(x,y,t) = 1 - 0.99 \sin(x + y - 2t), \quad
		u = 1, \quad
		v = 1, \quad
		p = 0.01, \quad
		Y = 0
		\]
		and the spatial domain is $[0,2\pi]\times[0,2\pi]$. Periodic boundary conditions are imposed in both directions. 
		Table~\ref{tab:error_new_data} gives the errors at \(t=0.02\), which verify the designed order of convergence. 
		The monotone decay of the discrete energy is shown in Figure~\ref{fig:energy_periodic_2d}, demonstrating the energy stability of our scheme.
	\end{example}
	
	\begin{table}[H]
		\centering
		\begin{threeparttable}
			\caption{\textbf{Example~\ref{ex:2d_smooth}:} Error table for the 2D smooth problem.}
			\label{tab:error_new_data}
			\small
			\setlength{\tabcolsep}{7pt}
			\begin{tabular}{c cc cc cc}
				\toprule
				$\delta_x=\delta_y$ & $L^1$ error & Order & $L^2$ error & Order & $L^\infty$ error & 	Order \\
				\midrule
				$\pi/32$   & 6.00e-04 & -    & 2.67e-04 & -    & 2.34e-04 & -    \\
				$\pi/64$   & 8.34e-05 & 2.85 & 4.20e-05 & 2.67 & 4.81e-05 & 2.28 \\
				$\pi/128$  & 1.05e-05 & 2.99 & 5.65e-06 & 2.89 & 7.08e-06 & 2.76 \\
				$\pi/256$  & 1.33e-06 & 2.98 & 7.20e-07 & 2.97 & 9.18e-07 & 2.95 \\
				$\pi/512$  & 1.66e-07 & 3.00 & 9.03e-08 & 2.99 & 1.16e-07 & 2.98 \\
				\bottomrule
			\end{tabular}
		\end{threeparttable}
	\end{table}
	
	\begin{figure}[H]
		\centering
		\includegraphics[width=0.6\textwidth]{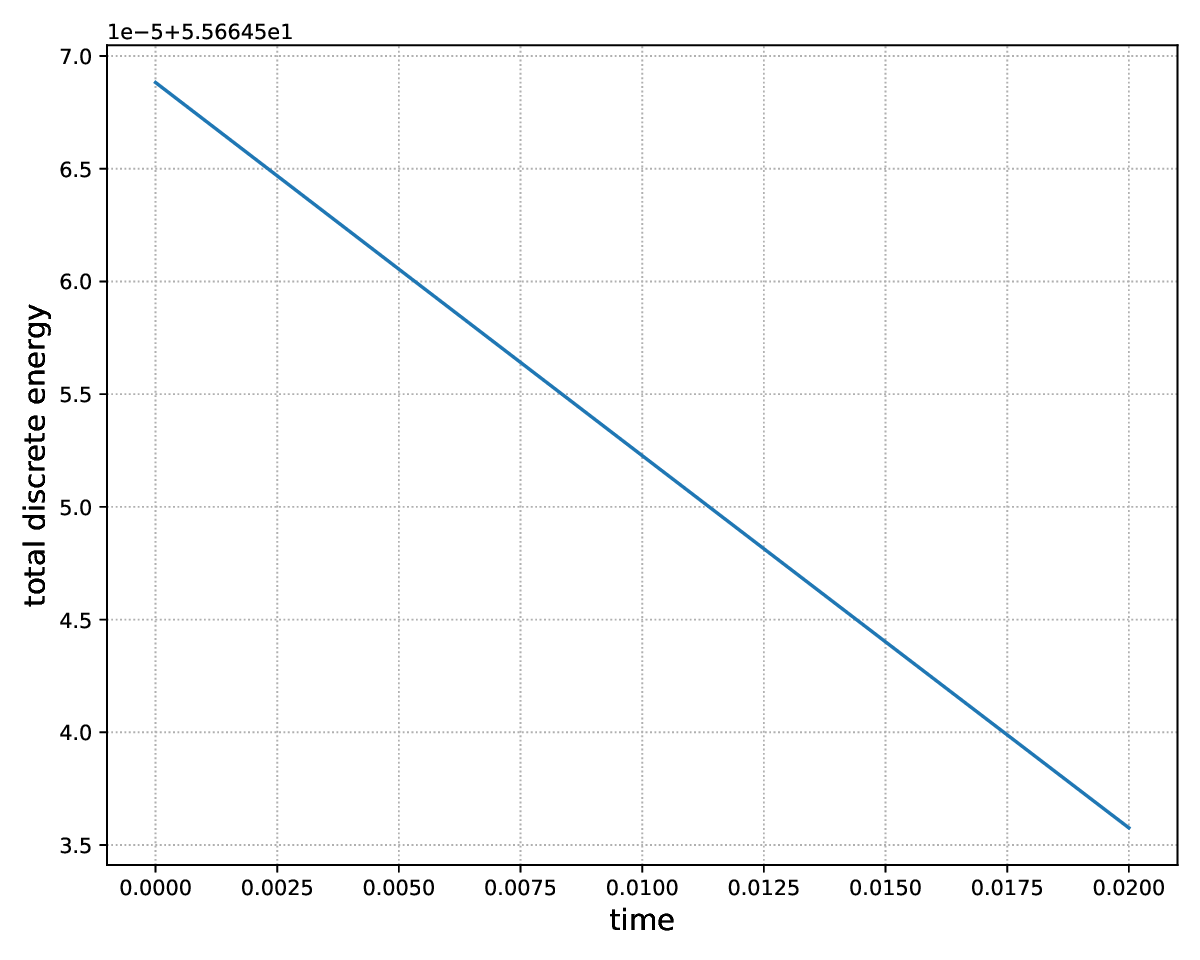}
		\caption{Example~\ref{ex:2d_smooth}: Discrete energy history for the 2D periodic problem.}
		\label{fig:energy_periodic_2d}
	\end{figure}
	
	\begin{example}\label{ex:riemann_case6}
		We next consider a 2D Riemann problem introduced in \cite{SchulzRinneCollinsGlaz1993Riemann2D}.
This is a non-reactive test, and we set $Y=0$, $\omega=0$, and $q=K=T_a=\sigma=0$. The square domain $[0,1]\times[0,1]$ is divided into four quadrants by $x=1/2$ and $y=1/2$, and the initial data are given by
		\[
		(\rho,u,v,p,Y)=
		\begin{cases}
			(2,\ 0.75,\ 0.5,\ 1,\ 0), & x<1/2,\ y\geq 1/2,\\
			(1,\ 0.75,\!-0.5,\ 1,\ 0), & x\geq 1/2,\ y\geq 1/2,\\
			(1,\!-0.75,\ 0.5,\ 1,\ 0), & x<1/2,\ y<1/2,\\
			(3,\!-0.75,\!-0.5,\ 1,\ 0), & x\geq 1/2,\ y<1/2.
		\end{cases}
		\]
We take $\gamma=1.4$, impose transmissive boundary conditions, and use $800\times800$ grid points, corresponding to $\delta_x=\delta_y=1/800$. The numerical solution is advanced to $t=0.3$ with $\mathrm{CFL}=0.8$. For the computation of the Navier--Stokes equations, we take $\mathrm{Re}=1000$, $\Pr=0.72$, $\mu=1/\mathrm{Re}$, $\lambda=-2\mu/3$, $\kappa=\mu c_p/\Pr$, and $c_p=\gamma/(\gamma-1)$.

		Figure~\ref{fig:riemann_case6_results} shows the pressure field by color, with density contours and velocity arrows superimposed. The four slip lines and their central interaction are clearly visible in both results. The results for the Euler equations are in qualitative agreement with those reported in \cite{LiskaWendroff2003Comparison}. The results of the Navier--Stokes equations retain the same overall wave pattern, while viscous and thermal diffusion broaden the contact layers. 
	\end{example}
	
	\begin{figure}[htbp]
		\centering
		\begin{subfigure}[b]{0.49\textwidth}
			\centering
			\includegraphics[width=\textwidth]{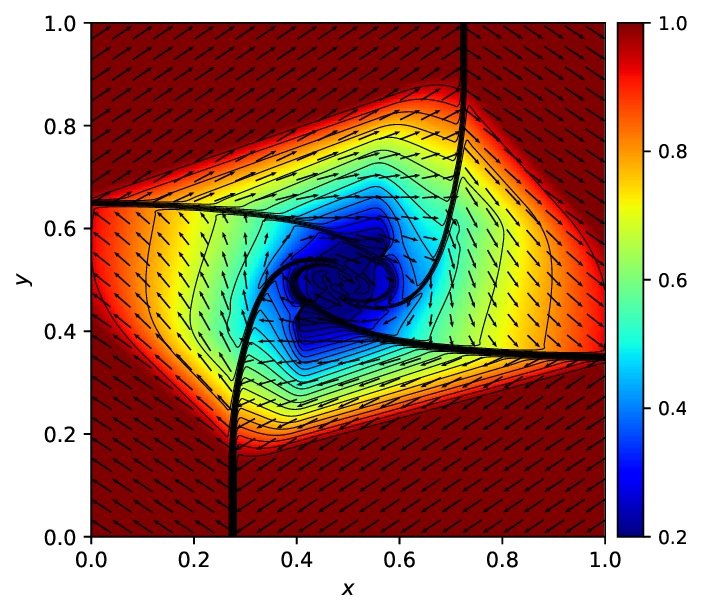}
			\caption{Euler}
		\end{subfigure}
		\hfill
		\begin{subfigure}[b]{0.49\textwidth}
			\centering
			\includegraphics[width=\textwidth]{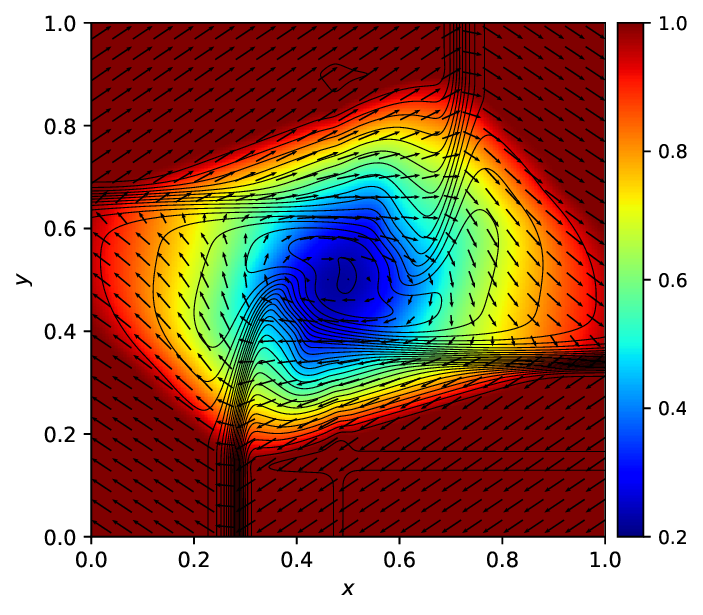}
			\caption{Navier--Stokes}
		\end{subfigure}
		\caption{Example~\ref{ex:riemann_case6}: Results of 2D Riemann problem at $t=0.3$. Pressure is displayed by color, density by 29 contours from $0.25$ to $3.05$ with an increment of $0.1$, and velocity by arrows. }
		\label{fig:riemann_case6_results}
	\end{figure}
	
	\clearpage
	\begin{example}\label{ex:yee_vortex}
		In this example, we simulate the isentropic vortex problem from \cite{YeeSandhamDjomehri1999Filters}, which has an exact solution for the Euler equations. Let $(\xi,\eta)=(x-u_\infty t, y-v_\infty t)$ be the displacement from the vortex center $(u_\infty t,v_\infty t)$ and denote $r^2=\xi^2+\eta^2$. The exact primitive variables are
		\[
		\begin{aligned}
			\theta&=1-\frac{(\gamma-1)\beta^2}{8\gamma\pi^2}\exp(1-r^2),\\
			\rho&=\rho_\infty\theta^{1/(\gamma-1)},\qquad
			p=p_\infty\left(\rho/\rho_\infty\right)^\gamma,\\
			u&=u_\infty-\frac{\beta}{2\pi}\eta\exp\left(\frac{1-r^2}{2}\right),\qquad
			v=v_\infty+\frac{\beta}{2\pi}\xi\exp\left(\frac{1-r^2}{2}\right).
		\end{aligned}
		\]
		As in \cite{YeeSandhamDjomehri1999Filters}, we set $\gamma=1.4$, $\beta=5$, $(\rho_\infty,u_\infty,v_\infty,p_\infty)=(1,1,1,1)$, and $Y=\omega=0$. The problem is solved on $[-5,5]^2$ to the final time $t=10$, with periodic boundary conditions. We use a $256\times256$ grid ($\delta_x=\delta_y=10/256$) and take $\mathrm{CFL}=0.8$. 
		For the computation of the Navier--Stokes equations, we use $\mathrm{Re}=500$ and $5000$, $\Pr=0.72$, $\mu=1/\mathrm{Re}$, $\lambda=-2\mu/3$, $\kappa=\mu c_p/\Pr$, $c_p=\gamma/(\gamma-1)$, and $\sigma=0$. Figures~\ref{fig:yee_vortex_fields} and~\ref{fig:yee_vortex_sections} show that the Euler solution preserves the transported vortex structure and that its density and pressure sections closely follow the exact solution. Viscous diffusion smooths the Navier--Stokes profiles, while increasing the Reynolds number from $500$ to $5000$ brings them closer to the exact Euler solution, indicating the expected approach to the inviscid limit. The exact Euler solution is used only as the inviscid reference, not as an exact viscous solution.
	\end{example}
	
	\begin{figure}[htbp]
		\centering
		\begin{subfigure}[b]{0.32\textwidth}
			\centering
			\includegraphics[width=\textwidth]{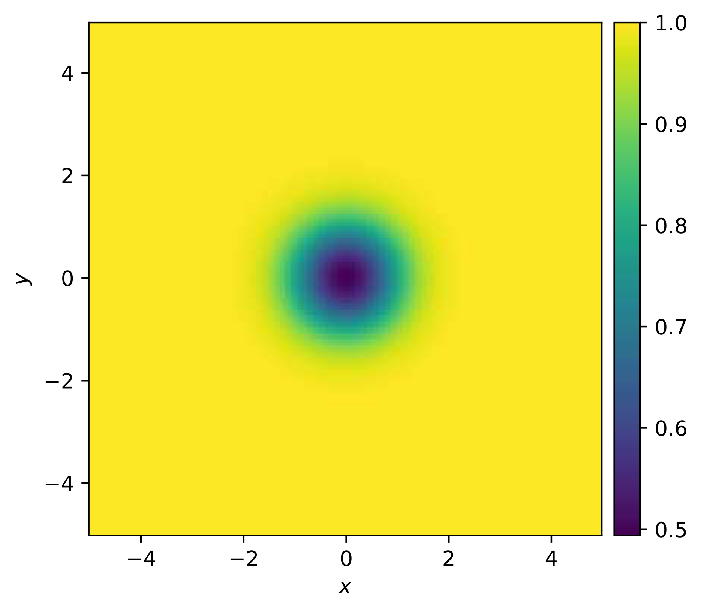}
			\caption{Density (Euler)}
		\end{subfigure}
		\hfill
		\begin{subfigure}[b]{0.32\textwidth}
			\centering
			\includegraphics[width=\textwidth]{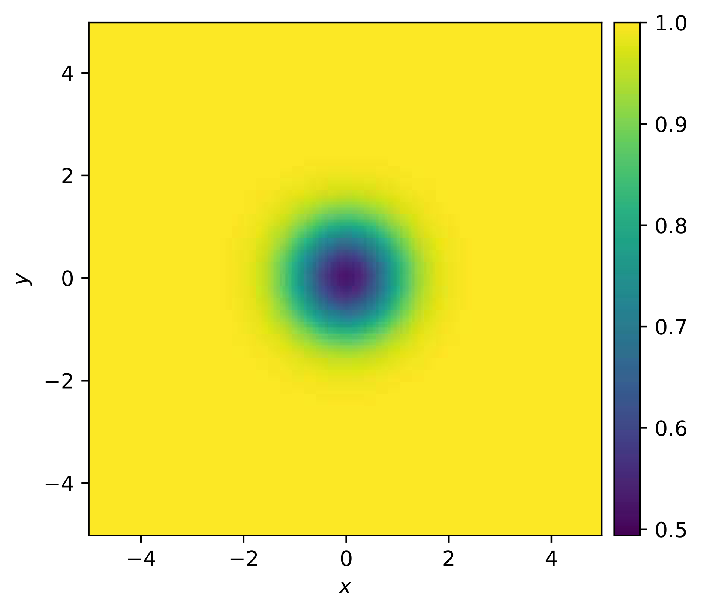}
			\caption{Density (NS, $\mathrm{Re}=500$)}
		\end{subfigure}
		\hfill
		\begin{subfigure}[b]{0.32\textwidth}
			\centering
			\includegraphics[width=\textwidth]{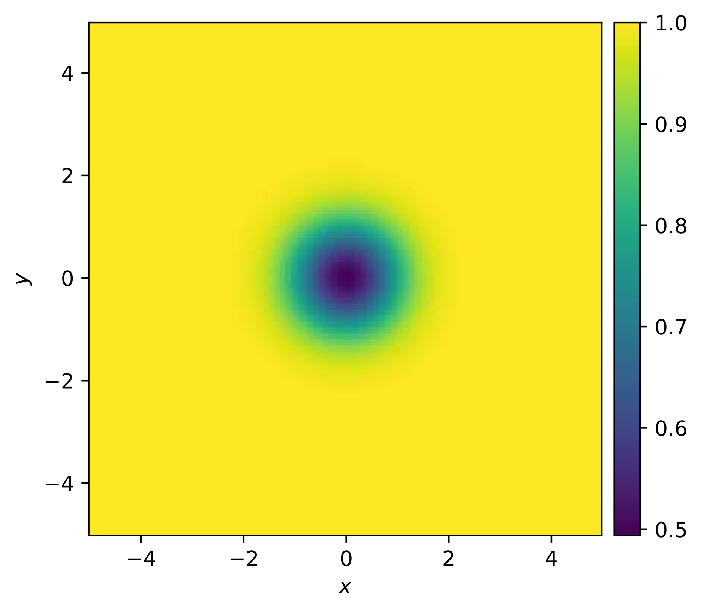}
			\caption{Density (NS, $\mathrm{Re}=5000$)}
		\end{subfigure}
		
		\vspace{0.3cm}
		
		\begin{subfigure}[b]{0.32\textwidth}
			\centering
			\includegraphics[width=\textwidth]{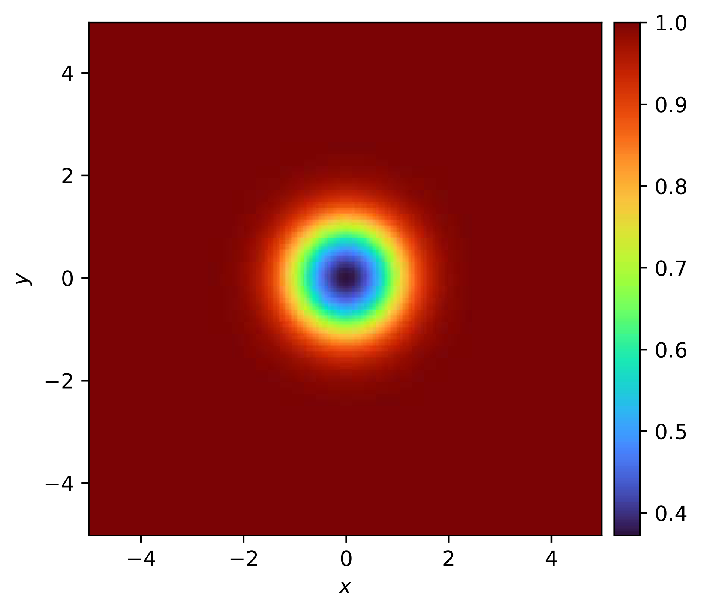}
			\caption{Pressure (Euler)}
		\end{subfigure}
		\hfill
		\begin{subfigure}[b]{0.32\textwidth}
			\centering
			\includegraphics[width=\textwidth]{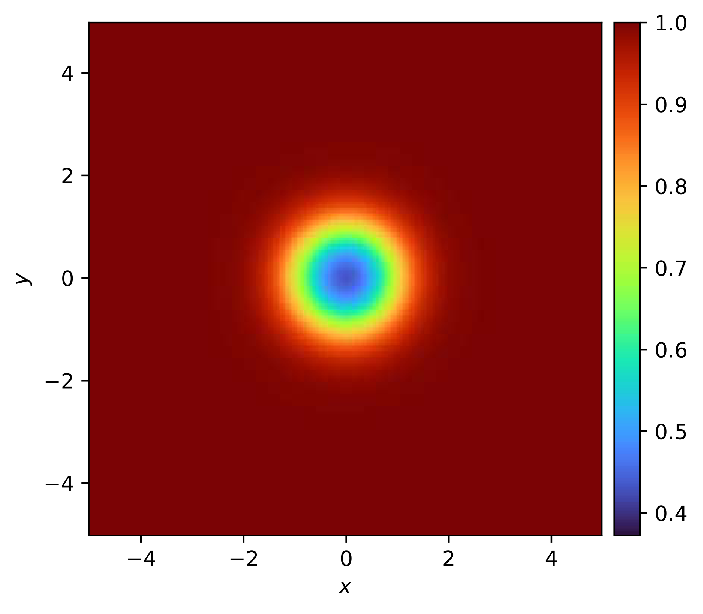}
			\caption{Pressure (NS, $\mathrm{Re}=500$)}
		\end{subfigure}
		\hfill
		\begin{subfigure}[b]{0.32\textwidth}
			\centering
			\includegraphics[width=\textwidth]{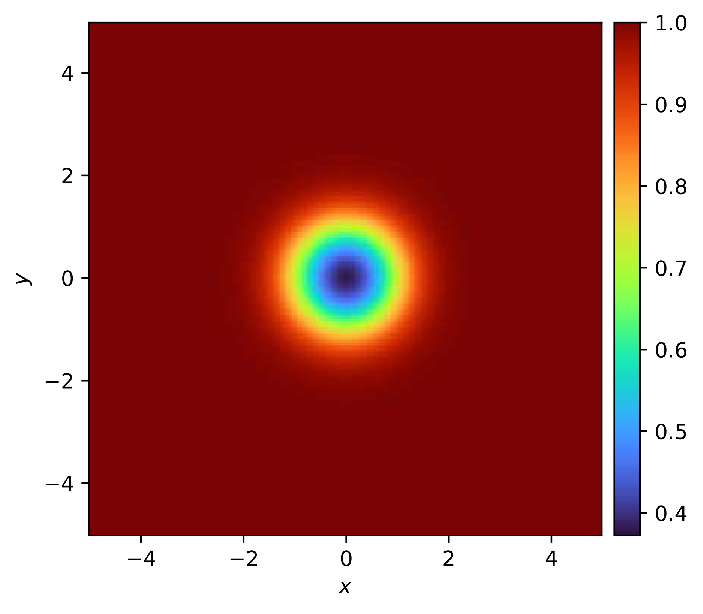}
			\caption{Pressure (NS, $\mathrm{Re}=5000$)}
		\end{subfigure}
		
		\caption{Example~\ref{ex:yee_vortex}: Density and pressure fields for the 2D isentropic vortex problem at $t=10$ on a $256\times256$ periodic grid.}
		\label{fig:yee_vortex_fields}
	\end{figure}
	
	\begin{figure}[htbp]
		\centering
		\begin{subfigure}[b]{0.49\textwidth}
			\centering
			\includegraphics[width=\textwidth]{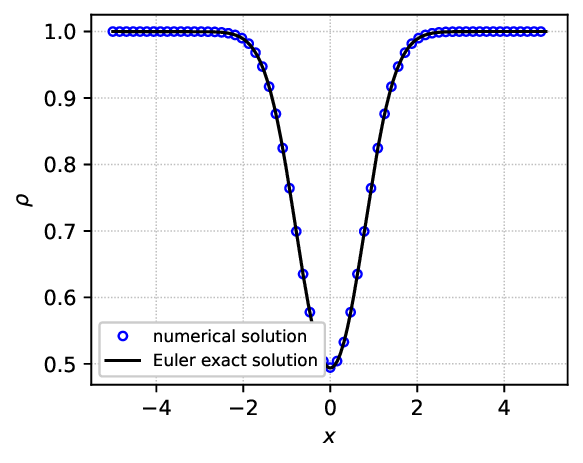}
			\caption{Density profile along $y=0$ (Euler)}
		\end{subfigure}
		\hfill
		\begin{subfigure}[b]{0.49\textwidth}
			\centering
			\includegraphics[width=\textwidth]{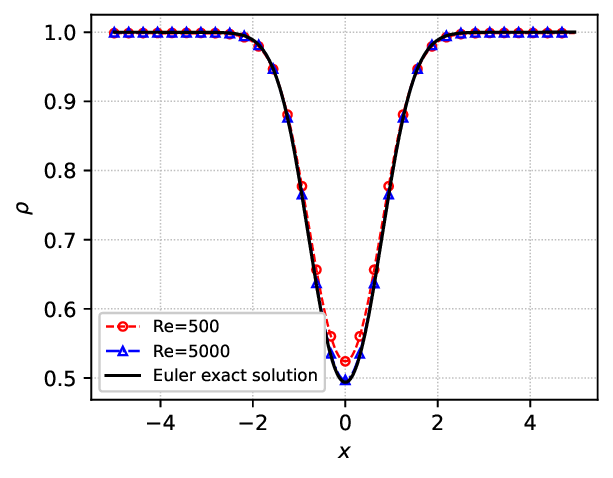}
			\caption{Density profile along $y=0$ (NS)}
		\end{subfigure}
		
		\vspace{0.3cm}
		
		\begin{subfigure}[b]{0.49\textwidth}
			\centering
			\includegraphics[width=\textwidth]{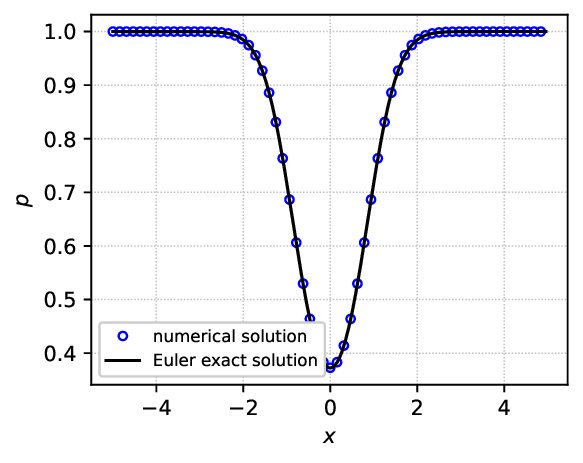}
			\caption{Pressure profile along $y=0$ (Euler)}
		\end{subfigure}
		\hfill
		\begin{subfigure}[b]{0.49\textwidth}
			\centering
			\includegraphics[width=\textwidth]{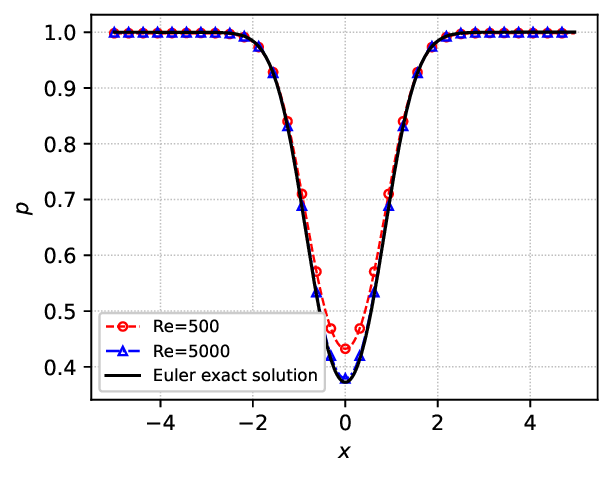}
			\caption{Pressure profile along $y=0$ (NS)}
		\end{subfigure}
		
		\caption{Example~\ref{ex:yee_vortex}: Density and pressure profiles along $y=0$ at $t=10$.}
		\label{fig:yee_vortex_sections}
	\end{figure}
	
	\begin{example}\label{ex:2d_radial}
		Next, we consider a 2D radially symmetric detonation-type problem \cite{WangZhangShuNing2012DetonationDG}. The computational domain is $[0,2]\times[0,2]$, and the initial data are
		\[
		(\rho, u, v, p, Y) =
		\begin{cases}
			(1, 0, 0, 80, 0.2), & x^2 + y^2 \le 0.36, \\
			(1, 0, 0, 10, 0.8), & otherwise.
		\end{cases}
		\]
		The present setup is a modified version of \cite{WangZhangShuNing2012DetonationDG}. The initial pressure and mass-fraction values are chosen so that the transformed variables remain real and the computation stays in the positivity range required by the present formulation. The reactive parameters are
		\(\gamma=1.2\), \(q=50\), \(K=2566.4\), and \(T_a=50\).
		Following the setup used in \cite{WangZhangShuNing2012DetonationDG,ZuoZhaoLin2024EntropyStableDGReactiveEuler}, solid-wall boundary conditions are imposed at the left and bottom boundaries. The solid walls are implemented by a mirror-reflective ghost-cell treatment: the normal component of the velocity is reversed across the wall, while the tangential velocity, density, pressure, and mass fraction are copied from the interior state. We set \(\mathrm{CFL}=0.8\) and compute the solution to \(t=0.16\). For the Navier--Stokes equations, we use \(\mathrm{Re}=200\), \(\Pr=0.72\), \(\mu=1/\mathrm{Re}\), \(\lambda=-2\mu/3\), \(\kappa=\mu c_p/\Pr\), \(c_p=\gamma/(\gamma-1)\), and \(\sigma=0.5\mu\). Figure~\ref{fig:2d_results1} shows the density and pressure contours, from which we see that the main circular wave structures are captured in both cases, while the viscous solution is smoother, as expected.
	\end{example}
	
	\begin{figure}[htbp]
		\centering
		\begin{subfigure}[b]{0.49\textwidth}
			\centering
			\includegraphics[width=\textwidth]{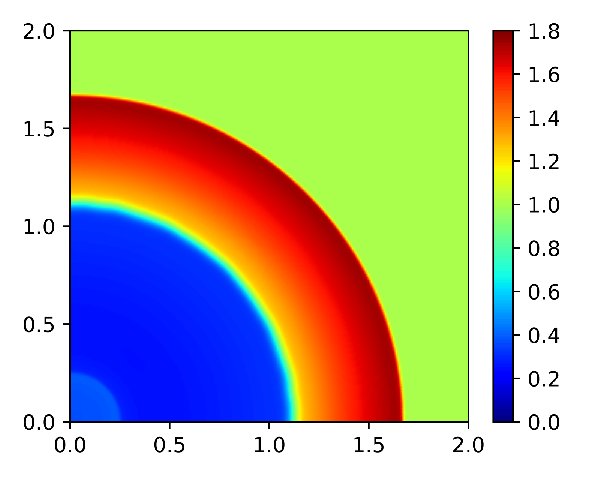}
			\caption{Density (Euler)}
		\end{subfigure}
		\hfill
		\begin{subfigure}[b]{0.49\textwidth}
			\centering
			\includegraphics[width=\textwidth]{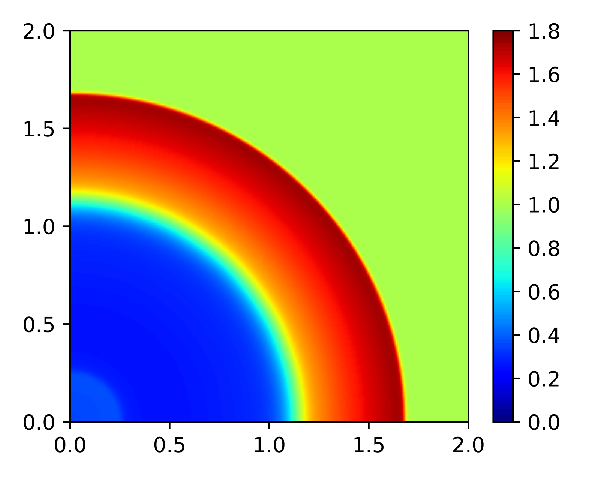}
			\caption{Density (NS)}
		\end{subfigure}
		
		\vspace{0.3cm}
		
		\begin{subfigure}[b]{0.49\textwidth}
			\centering
			\includegraphics[width=\textwidth]{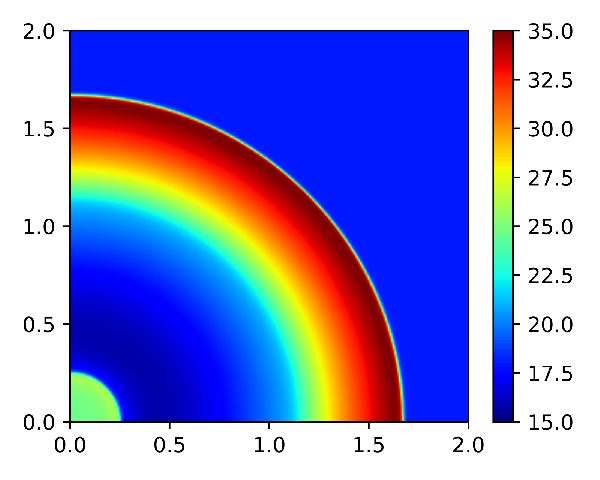}
			\caption{Pressure (Euler)}
		\end{subfigure}
		\hfill
		\begin{subfigure}[b]{0.49\textwidth}
			\centering
			\includegraphics[width=\textwidth]{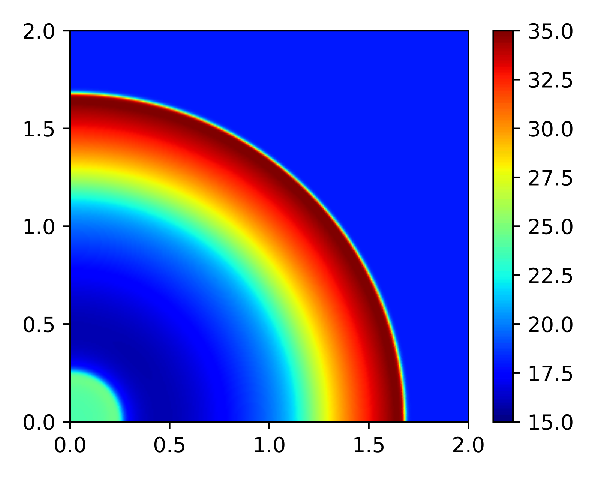}
			\caption{Pressure (NS)}
		\end{subfigure}
		
		\caption{Example~\ref{ex:2d_radial}: Contour plots of density and pressure at $t=0.16$ with $\delta_x=\delta_y=\frac{1}{100}$. $\mathrm{Re}=200$ for the Navier--Stokes equations.}
		\label{fig:2d_results1}
	\end{figure}
	
	\begin{example}\label{ex:2d_obstacles}
		Finally, we test our scheme for the reactive Euler equations by a detonation problem with multiple obstacles \cite{HuangShu2019PPTime,WangZhangShuNing2012DetonationDG,ZuoZhaoLin2024EntropyStableDGReactiveEuler}.
The computational domain is $[0,10]\times[0,10]$, with two solid obstacles located in $[1,3]\times[0,3]$ and $[5,10]\times[0,5]$. The initial data are
		\[
		(\rho, u, v, p, Y) =
		\begin{cases}
			(1, 0, 0, 80, 0.2), & x^2 + y^2 \le 0.36, \\
			(1, 0, 0, 10, 0.8), & otherwise.
		\end{cases}
		\]
		The Arrhenius reaction source is used with
		\(\gamma=1.2\), \(q=50\), \(K=2566.4\), and \(T_a=50\).
	Following \cite{HuangShu2019PPTime,WangZhangShuNing2012DetonationDG,ZuoZhaoLin2024EntropyStableDGReactiveEuler}, all outer boundaries and obstacle boundaries are treated as solid walls. The solid walls are implemented by a mirror-reflective ghost-cell treatment: the normal component of the velocity is reversed across the wall, while the tangential velocity, density, pressure, and mass fraction are copied from the interior state. We set $\delta_x=\delta_y=1/20$ and \(\mathrm{CFL}=0.8\). Figure~\ref{fig:2dresults2} shows the density and pressure contours at \(t=1\) and \(t=2\). It can be seen that our scheme remains stable and captures the main reflected and diffracted wave structures generated by the obstacles, which are consistent with the results in \cite{HuangShu2019PPTime,ZuoZhaoLin2024EntropyStableDGReactiveEuler}.
	\end{example}
	
	\begin{figure}[htbp]
		\centering
		\begin{subfigure}[b]{0.49\textwidth}
			\centering
			\includegraphics[width=\textwidth]{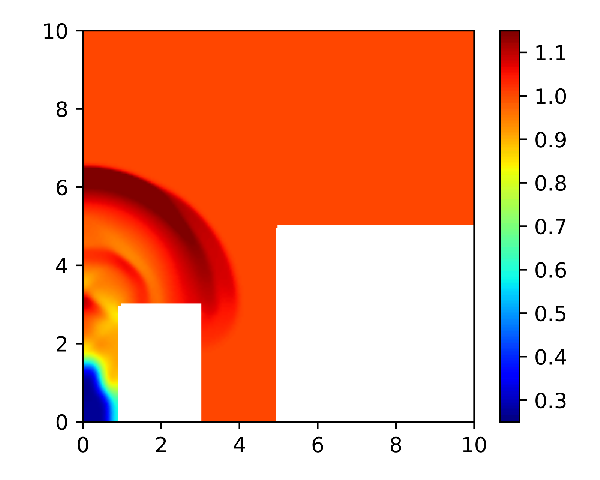}
			\caption{Density ($t=1$)}
		\end{subfigure}
		\hfill
		\begin{subfigure}[b]{0.49\textwidth}
			\centering
			\includegraphics[width=\textwidth]{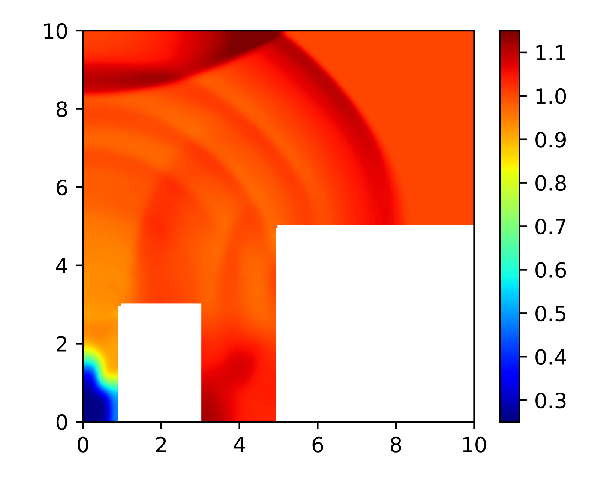}
			\caption{Density ($t=2$)}
		\end{subfigure}
		
		\vspace{0.3cm}
		
		\begin{subfigure}[b]{0.49\textwidth}
			\centering
			\includegraphics[width=\textwidth]{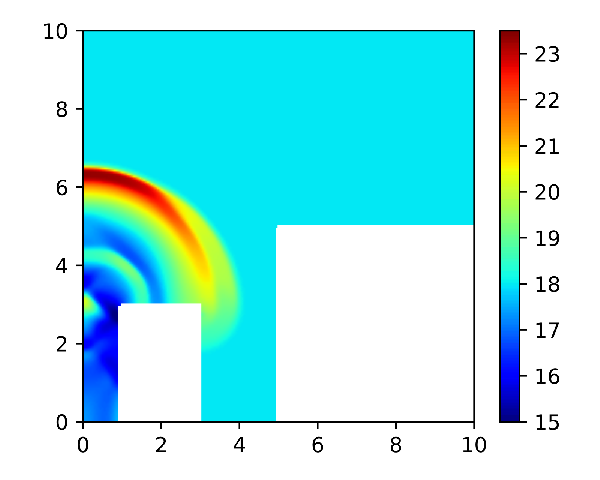}
			\caption{Pressure ($t=1$)}
		\end{subfigure}
		\hfill
		\begin{subfigure}[b]{0.49\textwidth}
			\centering
			\includegraphics[width=\textwidth]{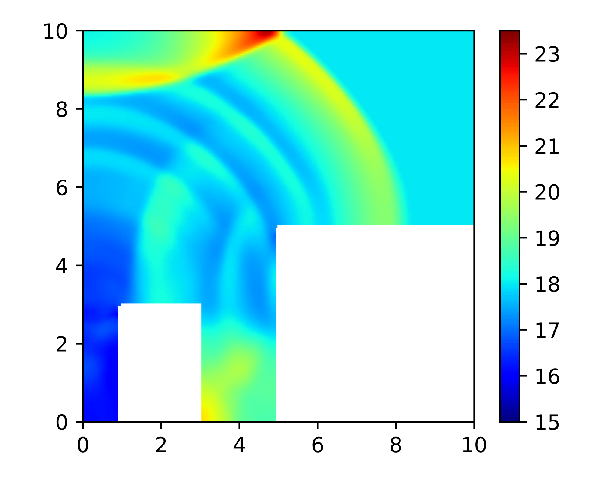}
			\caption{Pressure ($t=2$)}
		\end{subfigure}
		
		\caption{Example~\ref{ex:2d_obstacles}: Contour plots of density and pressure for the reactive  Euler equations at $t=1,2$ with $\delta_x=\delta_y=\frac{1}{20}$.}
		\label{fig:2dresults2}
	\end{figure}

	\section{Conclusions and remarks}
	\setcounter{equation}{0}
	\label{sec:conclusion}
	In this paper, we developed skew-symmetric formulations and energy stable finite difference schemes for the reactive Euler and reactive Navier--Stokes equations. Suitable transformed variables were introduced so that the convective parts of the governing equations can be written in skew-symmetric form. With this structure, continuous energy estimates were obtained for both inviscid and viscous reactive systems. The reaction term, the chemical energy, and the species diffusion were incorporated into the analysis in a way consistent with the sign of the energy estimate.
	
	Based on the continuous formulations, we constructed semi-discrete schemes using SBP operators. The non-dissipative part of linear upwind SBP operators was used to approximate the skew-symmetric convective terms, while the nonlinear SBP--WENO operator was employed only as an artificial dissipation mechanism. Under periodic boundary conditions, the resulting semi-discrete schemes satisfy energy estimates analogous to the continuous ones. Numerical experiments in one and two space dimensions verified the expected accuracy for smooth problems and demonstrated stable behavior for a shock tube, a double rarefaction, a reactive shock--wave interaction, two-dimensional contact--slip interactions, periodic vortex convection, and detonation-type tests.
	
	The present analysis is restricted to periodic boundary conditions for the discrete energy estimates. Some non-periodic and obstacle problems were included to assess the practical performance of the method, but a complete stability theory for such boundary treatments is left for future work. Further topics include the extension to three-dimensional computations, the construction of energy stable non-periodic boundary conditions for reactive flows, and the application of the skew-symmetric framework to more general multicomponent reacting systems.
	
	\bibliographystyle{elsarticle-num}
	\bibliography{ESS_ReactiveEuler}

\end{document}